\documentclass[10pt]{amsart}
\usepackage{amssymb}
\usepackage{amsmath}
\usepackage{amsthm}

\makeatletter
\def\LaTeX{\leavevmode L\raise.42ex
    \hbox{\kern-.3em\size{\sf@size}{0pt}\selectfont A}\kern-.15em\TeX}
\makeatother

\newcommand{\BibTeX}{{\rm B\kern-.05em{\sc
          i\kern-.025emb}\kern-.08em\TeX}}

\makeatletter
\def\@currentlabel{2.1}\label{e:dispaa}
\def\@currentlabel{2.21}\label{e:dispau}
\def\@currentlabel{2.22}\label{e:dispav}
\def\@currentlabel{2.23}\label{e:dispaw}
\def\@currentlabel{2.24}\label{e:dispax}
\def\theequation{\thesection.\@arabic\c@equation}
\makeatother

\renewcommand{\theequation}{\arabic{section}.\arabic{equation}}

\newtheorem{thm}{Theorem}[section]
\newtheorem{lem}[thm]{Lemma}
\newtheorem{cor}[thm]{Corollary}
\newtheorem{prop}[thm]{Proposition}
\theoremstyle{definition}
\newtheorem{defn}{Definition}
\newtheorem{rem}[thm]{Remark}

\newcommand{\R}{\mathbb{R}}

\title[Weighted elliptic equations]
{Embeddings of weighted Sobolev spaces and a weighted fourth order elliptic equation}

\author{Zongming Guo}
\address{Department of Mathematics, Henan Normal University,
Xinxiang, 453007, P.R. China} \email{gzm@htu.cn}

\author{Fangshu Wan}
\address{Department of Mathematics, Henan Normal University,
Xinxiang, 453007, P.R. China}
\email{fangshuwan@163.com}

\author{Liping Wang*}
\address{Department of Mathematics, Shanghai Key Laboratory of Pure Mathematics and Mathematical Practice, East China Normal University, 200241, P.R. China}
\email{lpwang@math.ecnu.edu.cn}

\thanks{*: corresponding author.  \\
The research of the first author is supported by NSFC
11571093 and the third author is supported by NSFC 11671144.}

\date{}

\begin{document}

\maketitle
\begin{abstract}
New embeddings of weighted Sobolev spaces are established. Using such embeddings, we obtain the existence and regularity of positive solutions with Navier boundary value problems for a weighted fourth order elliptic equation. We also obtain
Liouville type results for the related equation. Some problems are still open.
\end{abstract}



\maketitle \baselineskip 18pt
\section{Introduction}
\setcounter{equation}{0}

We study structure of nonnegative solutions of the weighted fourth order elliptic equation
$$ \Delta (|x|^\alpha \Delta u)=|x|^l u^p \;\;\mbox{in $\R^N$}, \leqno(P)$$
where $N \geq 5$, $p>1$, $\alpha, l \in \R$.

For a bounded smooth domain $\Omega \subset \R^N$ and $0 \in \Omega$, we also study the Navier boundary value problem
$$\left \{ \begin{array}{ll} \Delta (|x|^\alpha \Delta u)=|x|^l u^p \;\; & \mbox{in $\Omega$},\\
 u=\Delta u=0 \;\; &\mbox{on $\partial \Omega$}.
 \end{array} \right. \leqno(Q)$$

By a positive solution of (P), we mean that $u \in C^4 (\R^N \backslash \{0\}) \cap C^0 (\R^N)$,
$|x|^\alpha \Delta u \in C^0 (\R^N)$,
$u>0$ in $\R^N$ and $u$ satisfies (P) in $\R^N \backslash \{0\}$. By a positive solution of (Q), we mean that $u \in C^4 (\Omega \backslash \{0\}) \cap C^0 ({\overline \Omega})$, $|x|^\alpha \Delta u \in C^0 ({\overline \Omega})$, $u>0$ in $\Omega$ and $u$ satisfies (Q) in $\Omega \backslash \{0\}$.

Throughout this paper, we assume
\begin{equation}
\label{1.1}
2N>N':=N+\alpha>4, \;\;\; \tau:=l-\alpha>-4,
\end{equation}
\begin{equation}
\label{1.1-1}
p_s:=\frac{N'+4+2 \tau}{N'-4} \; (>1).
\end{equation}
It is known from \eqref{1.1} that $\alpha \in (4-N,N)$.

Equation (P) and problem (Q) have been studied by many authors recently, in particular, in the case of the pure biharmonic operator, see, for example, \cite{AGGM, BFFG, CC, CM, CEGM, DDGM, DGGW, EFJ, GG, Guo, GHZ, GW1, GW2, GW3, Lin, Mu} and the references therein.

Under the assumptions in \eqref{1.1}, the following Liouville theorem for (P) is
still  open.

{\it {\bf Open problem 1:} Assume that $N \geq 5$ and \eqref{1.1} holds, $u \in C^4 (\R^N \backslash \{0\}) \cap C^0 (\R^N)$ and $|x|^\alpha \Delta u \in C^0 (\R^N)$ is a nonnegative solution of (P). Then
$u \equiv 0$ in $\R^N$ provided $1<p<p_s$}.

 The positive answer to this open problem for $\alpha=l=0$ was given in \cite{Lin}  via the moving plane argument. But it does not apply for our case here (especially for $l>0$), since the weights do not match with the procedure of this argument. In this paper, we will give a partial answer to this open problem. More precisely,  when $\alpha \in (4-N, N)$, we obtain the positive answer if $1<p<\min \{\frac{N+4}{N-4}, p_s\}$. Especially, when  $\alpha=2$ we confirm the result  for $1<p<p_s$. Namely, we have the following theorems.

\begin{thm}
\label{t0.1}
Assume $N \geq 5$ and \eqref{1.1} holds. If $u \in C^4 (\R^N \backslash \{0\})
\cap C^0 (\R^N)$ and $|x|^\alpha \Delta u \in C^0 (\R^N)$ is a nonnegative solution to (P). Then $u \equiv 0$ in $\R^N$ provided $1<p<\min \{\frac{N+4}{N-4}, p_s \}$.
\end{thm}
Note that $p_s \leq \frac{N+4}{N-4}$ is equal to $\alpha \geq \frac{(N-4)}{4} \tau$ and $p_s>\frac{N+4}{N-4}$ provided $\alpha<\frac{(N-4)}{4} \tau$.  Hence Theorem \ref{t0.1}  implies that
 if $\alpha \geq \frac{(N-4)}{4} \tau$,  the answer is positive  for $1<p<p_s$  while $\alpha<\frac{(N-4)}{4} \tau$, it only gives out the result for   $1<p<\frac{N+4}{N-4}$.  So there is a gap  $\frac{N+4}{N-4} \leq p<p_s$ provided $\alpha<\frac{(N-4)}{4} \tau$. But for the symmetric solutions or $\alpha =2$, we can still get the result in  Theorem \ref{t0.1}.

\begin{thm}
\label{t0.2}
Assume $N \geq 5$ and \eqref{1.1} holds. Assume also that $u \in C^4 (\R^N \backslash \{0\})
\cap C^0 (\R^N)$ with $|x|^\alpha \Delta u \in C^0 (\R^N)$ is a nonnegative radial solution to (P). Then $u \equiv 0$ in $\R^N$ provided $1<p<p_s$.
\end{thm}

\begin{thm}
\label{t1.1}
Assume that $N \geq 5$, $\alpha=2$ and $l>-2$. Assume also that $u \in C^4 (\R^N \backslash \{0\})
\cap C^0 (\R^N)$ and $|x|^2 \Delta u \in C^0 (\R^N)$ is a nonnegative solution to (P). Then $u \equiv 0$ in $\R^N$ provided $1<p<p_s=\frac{N+2+2l}{N-2}$.
\end{thm}
\begin{rem}
\label{r4.1}
There is no first order term in (\ref{3.12}) if $\alpha=2$ and we can  get the monotonicity in the radial direction.

\end{rem}

For a bounded domain $\Omega \subset \R^N$ of class $C^2$ with $0 \in \Omega$, we
define $D^{1, \alpha}_0 (\Omega)$ as the completion of $C_c^\infty (\Omega \backslash \{0\})$
with  the  norm
$$\|\phi\|_{1, \alpha}^2=\int_\Omega |x|^\alpha |\nabla \phi|^2 dx$$
and  $D^{2, \alpha}_0 (\Omega)$ as the completion of $C_c^\infty (\Omega \backslash \{0\})$
with the  norm
$$\|\phi\|_{2, \alpha}^2=\int_\Omega |x|^\alpha |\Delta \phi|^2 dx.$$
Also we define $L_l^q (\Omega) \; (q \geq 1)$ the space of functions $\phi$ such that
$$|x|^\frac{l}{q} |\phi| \in L^q (\Omega)$$
with the norm
$$\|\phi\|_{L_l^q (\Omega)}=\Big(\int_\Omega |x|^l |\phi|^q dx \Big)^{\frac{1}{q}}.$$

We now establish the following embeddings:

\begin{prop}
\label{p1.2}
Let $\Omega \subset \R^N$ be a bounded domain of class $C^2$, but not a ball, with $0 \in \Omega$. Assume also that $N \geq 5$ and \eqref{1.1} holds with $\alpha \geq \frac{(N-4)}{4} \tau$. Then, the embedding:
\begin{equation}
\label{1.2}
D^{2, \alpha}_0 (\Omega) \hookrightarrow L_l^{q} (\Omega)
\end{equation}
is continuous for any $q \in [1, p_s+1]$ and this embedding is  compact for any $q \in [1, p_s+1)$.
\end{prop}
Note that $p_s \leq \frac{N+4}{N-4}$ provided $\alpha \geq \frac{(N-4)}{4} \tau$. The related embeddings from $D^{1,\alpha}_0 (\Omega)$ to $L_l^q (\Omega)$ was obtained in \cite{GGW}. Such embeddings may have been known already, but we can not find a suitable reference.

\begin{prop}
\label{p1.3}
Let $\Omega \subset \R^N$ be a bounded domain of class $C^2$, but not a ball, with $0 \in \Omega$. Assume also that $N \geq 5$ and \eqref{1.1} holds with $\alpha<\frac{(N-4)}{4} \tau$. Then, the embedding:
\begin{equation}
\label{1.3}
D^{2, \alpha}_0 (\Omega) \hookrightarrow L_l^{q} (\Omega)
\end{equation}
is continuous for any $q \in [1, \frac{2N}{N-4}]$ and this embedding is  compact for any $q \in [1, \frac{2N}{N-4})$.
\end{prop}
Note that $p_s>\frac{N+4}{N-4}$ provided $\alpha<\frac{(N-4)}{4} \tau$. So the above two propositions establish the embedding  theorem for $1<p<\min\{\frac{N+4}{N-4}, p_s\}$. Actually, for radial functions in the ball, we can achieve better embedding theorem.

\begin{prop}
\label{p1.4}
Let $B$ be a ball centered at 0. Assume also that $N \geq 5$ and \eqref{1.1} holds. Then, the embedding:
$$D^{2, \alpha}_{0, \mbox{{\tiny rad}}} (B) \hookrightarrow L_l^{q} (B)$$
is continuous for any $q \in [1, p_s+1]$ and this embedding is  compact for any $q \in [1, p_s+1)$, where $D^{2, \alpha}_{0, \mbox{{\tiny rad}}} (B)=\{\phi \in D_0^{2, \alpha} (B): \;
\phi(x)=\phi (|x|)\}$.
\end{prop}

Using the embeddings given in Propositions \ref{p1.2}-\ref{p1.4}, we can obtain positive solutions $u$ of (Q) or positive radial solutions $u$ of (Q) via variational methods for $1<p<p_s$ or $1<p<\frac{N+4}{N-4}$.
The existence of positive radial solutions of the related weighted elliptic equations for second order was obtained in \cite{Ni}.
The boot-strap argument implies that the obtained solutions $u$ of (Q) satisfy $u \in C^4 (\Omega \backslash \{0\})$ or $u \in C^4 (B \backslash \{0\})$. We will see from the theorems below that the singular point 0 of $u$ is removable, i.e., $u \in C^4 (\Omega \backslash \{0\})
\cap C^0 ({\overline \Omega})$ and $|x|^\alpha \Delta u \in C^0 ({\overline \Omega})$ or $u \in C^4 (B \backslash \{0\}) \cap C^0 ({\overline B})$ and $|x|^\alpha \Delta u \in C^0 ({\overline B})$.

We know that $p_s>\frac{N+4}{N-4}$ provided $\alpha<\frac{N-4}{4} \tau$. The embedding in
Proposition \ref{p1.3} implies that we can only obtain the existence and regularity of positive solutions of (Q) for $1<p<\frac{N+4}{N-4}$. The following problem is still open.

{\it {\bf Open problem 2.} Assume that $\Omega$ is a bounded smooth domain in $\R^N \; (N \geq 5)$ with $0 \in \Omega$ but  not a ball. Assume also that \eqref{1.1} holds with $\alpha<\frac{N-4}{4} \tau$.

Does (Q) admit a positive solution $u \in C^4 (\Omega \backslash \{0\}) \cap C^0 ({\overline \Omega})$ and $|x|^\alpha \Delta u \in C^0({\overline \Omega})$ provided $\frac{N+4}{N-4} \leq p< p_s$?}

In this paper, we use $C$ to denote a universal positive constant, which may be changed from one line to another line.

\section{Embeddings of weighted Sobolev spaces and proof of Propositions \ref{p1.2}-\ref{p1.4}}
\setcounter{equation}{0}

In this section, we present the proof of Propositions \ref{p1.2}-\ref{p1.4}.

{\bf Proof of Proposition \ref{p1.2}}

Since $\alpha \geq \frac{N-4}{4} \tau$, we see that $p_s \leq \frac{N+4}{N-4}$.
It is known from Theorem 1.1 and Remark 2.2 of \cite{CM} that the embedding:
\begin{equation}
\label{2.1}
D^{2, \alpha}_0 (\Omega) \hookrightarrow L_l^{p_s+1} (\Omega)
\end{equation}
is continuous. That is, there exists a constant $C>0$ independent of $u$ such that
\begin{equation}
\label{2.2}
\Big(\int_\Omega |x|^l |u|^{p_s+1} dx \Big)^{\frac{2}{p_s+1}} \leq C \int_\Omega |x|^\alpha (\Delta u)^2 dx.
\end{equation}
This inequality is corresponding to  Caffarelli-Kohn-Nirenberg inequality for the second order operator (see \cite{CKN, CW}). Note that when $4-N<\alpha<N$,
$$-\gamma_\alpha:=\Big(\frac{\alpha-2}{2}\Big)^2-\Big(\frac{N-2}{2} \Big)^2 \not \in \Lambda (S^{N-1}),$$
where $\Lambda (S^{N-1})$ is the Dirichlet spectrum of the Laplace-Beltrami operator on $S^{N-1}$.

We now show that the embedding in \eqref{1.2} is continuous for any $q \in [1, p_s+1)$. To see this, we just notice that, for $u \in D^{2, \alpha}_0 (\Omega)$ and $q \in [1, p_s+1)$,
\begin{eqnarray}
\label{2.3}
\Big(\int_\Omega |x|^l |u|^q dx \Big)^{\frac{1}{q}} & \leq & \Big(\int_\Omega |x|^l |u|^{p_s+1} dx \Big)^{\frac{1}{p_s+1}} \Big(\int_\Omega |x|^l dx \Big)^{\frac{1}{q}-\frac{1}{p_s+1}} \nonumber \\
&\leq& C \Big(\int_\Omega |x|^\alpha (\Delta u)^2 dx \Big)^{\frac{1}{2}} \Big(\int_\Omega |x|^l dx \Big)^{\frac{1}{q}-\frac{1}{p_s+1}} \;\;\;\;\; (\mbox{by \eqref{2.2}}) \nonumber \\
&\leq& C \Big(\int_\Omega |x|^\alpha (\Delta u)^2 dx \Big)^{\frac{1}{2}}.
\end{eqnarray}
This implies that the embedding is continuous. Note that $\int_\Omega |x|^l dx \leq C$ provided
$N+l=N'+\tau>0$.

To prove the embedding is compact, we divide the proof into 2 steps.

{\it Step 1.} We show that if $\{u_m\}_{m=1}^\infty$ is a bounded sequence in $D_0^{2, \alpha} (\Omega)$, there exists a subsequence (still denoted by $\{u_m\}_{m=1}^\infty$) such that $u_m$ converges in $L_l^1 (\Omega)$.

Suppose that $\{u_m\}_{m=1}^\infty$ is a bounded sequence in $D_0^{2, \alpha} (\Omega)$, then there exists $C>0$ independent of $m$ such that $\|u_m\|_{D_0^{2, \alpha} (\Omega)} \leq C$. Setting $v_m=|x|^l u_m$, we see from \eqref{2.3} that $v_m \in L^1 (\Omega)$ and $\|v_m\|_{L^1(\Omega)} \leq C$, where $C>0$ is independent of $m$. Let $B$ be a large ball such that $\Omega \subset \subset B$. If we denote
$${\tilde u}_m (x)=\left \{ \begin{array}{ll} u_m (x), \;\;\; & x \in \Omega, \\
0, \;\;\; & x \in B \backslash \Omega \end{array} \right.$$
and ${\tilde v}_m (x)=|x|^l {\tilde u}_m (x)$, we see that ${\tilde v}_m \in L^1 (B)$ and $\|{\tilde v}_m \|_{L^1 (B)} \leq C$, where $C>0$ is independent of $m$. Let
$${\tilde v}_m^\epsilon (x)=\epsilon^{-N} \int_B \rho \Big(\frac{x-y}{\epsilon} \Big) {\tilde v}_m (y) d y,$$
where $0<\epsilon<\delta:=$ dist$(\partial \Omega, \partial B)$ and $\rho \in C^\infty (\R^N)$, $\rho (x)=0$ for $|x| \geq 1$, $\int_{\R^N} \rho (x) dx=1$. Obviously, $\{{\tilde v}_m^\epsilon\}_{m=1}^\infty$ all have their supports in $B$. We now claim that for each fixed $\epsilon>0$, the sequence $\{{\tilde v}_m^\epsilon\}_{m=1}^\infty$ is uniformly bounded and equi-continuous. Indeed, we see that, for $x \in B$,
$$|{\tilde v}_m^\epsilon (x)| \leq \sup |\rho| \epsilon^{-N} \|{\tilde v}_m\|_{L^1 (B)} \leq C \epsilon^{-N},$$
$$|\nabla {\tilde v}_m^\epsilon (x)| \leq \sup |\nabla \rho| \epsilon^{-N-1} \|{\tilde v}_m \|_{L^1 (B)} \leq C \epsilon^{-N-1},$$
where $C>0$ is independent of $m$. It follows from Arzela-Ascoli theorem that $\{{\tilde v}_m^\epsilon\}_{m=1}^\infty$ is precompact in $C^0 ({\overline B})$. Therefore, it is precompact in $L^1 (B)$. We next claim that
\begin{equation}
\label{2.4}
{\tilde v}_m^\epsilon \to {\tilde v}_m \;\;\; \mbox{in $L^1 (B)$ as $\epsilon \to 0$ uniformly in $m$}.
\end{equation}
We see
\begin{eqnarray}
\label{2.5}
\|{\tilde v}_m^\epsilon-{\tilde v}_m\|_{L^1 (B)} &=& \|{\tilde v}_m^\epsilon-{\tilde v}_m\|_{L^1 (B_{2 \epsilon})}+\|{\tilde v}_m^\epsilon-{\tilde v}_m\|_{L^1 (B \backslash B_{2 \epsilon})}
\nonumber \\
&\leq& 2 \|{\tilde v}_m \|_{L^1 (B_{3 \epsilon})}
+\|{\tilde v}_m^\epsilon-{\tilde v}_m\|_{L^1 (B \backslash B_{2 \epsilon})} \nonumber \\
&\leq& C \epsilon^\eta+\|{\tilde v}_m^\epsilon-{\tilde v}_m\|_{L^1 (B \backslash B_{2 \epsilon})}
\end{eqnarray}
where $B_{2 \epsilon}=\{x: \; |x|<2 \epsilon\}$, $\eta=\Big[1-\frac{1}{p_s+1} \Big] (N'+\tau)>0$. Note that, for $q \in [1, p_s+1)$ and $v \in D_0^{2, \alpha} (\Omega)$,
\begin{eqnarray*}
\Big(\int_{B_{3 \epsilon}} |x|^l |v|^q dx \Big)^{\frac{1}{q}} &\leq&
\Big(\int_{B_{3 \epsilon}} |x|^l |v|^{p_s+1} dx \Big)^{\frac{1}{p_s+1}} \Big(\int_{B_{3 \epsilon}} |x|^l dx \Big)^{\frac{1}{q}-\frac{1}{p_s+1}} \\
&\leq& C \Big(\int_\Omega |x|^\alpha (\Delta v)^2 dx \Big)^{\frac{1}{2}} \epsilon^{(\frac{1}{q}-\frac{1}{p_s+1}) [N'+\tau]}.
\end{eqnarray*}
We also see that
\begin{eqnarray*}
& & \int_{B \backslash B_{2 \epsilon}} |{\tilde v}_m^\epsilon (x)-{\tilde v}_m (x)| dx \\
 \leq & &\int_{|z| \leq 1} \rho (z) \int_0^{1} \int_{B \backslash B_{2 \epsilon}} \Big| \nabla {\tilde v}_m (x-\epsilon t z)\cdot \epsilon z  \Big| dx dt dz\\
\leq & &\epsilon \int_{B \backslash B_\epsilon} |\nabla {\tilde v}_m (z)| dz\\
\leq & & C \epsilon \int_{B \backslash B_\epsilon} |x|^{l-1} |{\tilde u}_m (x) dx
+\epsilon \int_{B \backslash B_\epsilon} |x|^l |\nabla {\tilde u}_m (x)| dx\\
\leq & & C \epsilon  \Big(\int_{B \backslash B_\epsilon} |x|^l |{\tilde u}_m|^{p_s+1} dx \Big)^{\frac{1}{p_s+1}} \Big( \int_{B \backslash B_\epsilon} |x|^{l-\frac{2(N'+\tau)}{N'+4+2 \tau}} \Big)^{\frac{N'+4+2 \tau}{2(N'+\tau)}}\\
& &+\epsilon \Big(\int_{B \backslash B_\epsilon} |x|^{\alpha-2} |\nabla {\tilde u}_m (x)|^2 dx \Big)^{\frac{1}{2}} \Big(\int_{B \backslash B_\epsilon} |x|^{l+2+\tau} dx \Big)^{\frac{1}{2}}\\
\leq& & C \epsilon \Big(\int_B |x|^\alpha (\Delta {\tilde u}_m (x))^2
dx \Big)^{\frac{1}{2}} \Big( \int_{B \backslash B_\epsilon} |x|^{l-\frac{2(N'+\tau)}{N'+4+2 \tau}} \Big)^{\frac{N'+4+2 \tau}{2(N'+\tau)}}\\
& &+\epsilon \Big(\int_B |x|^\alpha (\Delta {\tilde u}_m (x))^2
dx \Big)^{\frac{1}{2}} \Big(\int_{B \backslash B_\epsilon} |x|^{l+2+\tau} dx \Big)^{\frac{1}{2}}\\
 \leq & & C \epsilon \|{\tilde u}_m \|_{D_0^{2, \alpha} (B)}
\Big[ \Big(\int_{B \backslash B_\epsilon} |x|^{l-\frac{2(N'+\tau)}{N'+4+2 \tau}} dx \Big)^{\frac{N'+4+2 \tau}{2(N'+\tau)}}+\Big(\int_{B \backslash B_\epsilon} |x|^{l+2+\tau} dx \Big)^{\frac{1}{2}} \Big],
\end{eqnarray*}
where we have used that the fact (see \cite{GM1, GM2, Mo})
\begin{equation}
\label{2.5}
\int_B |x|^{\alpha-2} |\nabla {\tilde u}_m |^2 dx \leq C \int_B |x|^\alpha (\Delta {\tilde u}_m )^2 dx,
\end{equation}
where $C>0$ is independent of $m$.
We consider two cases here: (i) $N'+2 \tau \geq 0$, (ii) $N'+2 \tau<0$.

For the case (i), we have that
$$\int_{B \backslash B_\epsilon} |x|^{l-\frac{2(N'+\tau)}{N'+4+2 \tau}}dx  \leq C,$$
$$\int_{B \backslash B_\epsilon} |x|^{l+2+\tau} dx \leq C,$$
where $C>0$ is independent of $\epsilon$, and hence
$$\int_{B \backslash B_{2 \epsilon}} |{\tilde v}_m^\epsilon (x)-{\tilde v}_m (x)| dx  \leq C \epsilon.$$

For the case (ii), we have that
$$ \epsilon \Big( \int_{B \backslash B_\epsilon} |x|^{l-\frac{2(N'+\tau)}{N'+4+2 \tau}} dx \Big)^{\frac{N'+4+2 \tau}{2(N'+\tau)}} \leq C \epsilon^{\frac{N'+2 \tau+4}{2}},$$
$$\epsilon \Big(\int_{B \backslash B_\epsilon} |x|^{l+2+\tau} dx \Big)^{\frac{1}{2}}
\leq C \epsilon^{\frac{N'+2 \tau+4}{2}},$$
where $C>0$ is independent of $\epsilon$, and hence
$$\int_{B \backslash B_{2 \epsilon}} |{\tilde v}_m^\epsilon (x)-{\tilde v}_m (x)| dx  \leq 2 C \epsilon^{\eta_2}\|{\tilde u}_\epsilon \|_{D_0^{2, \alpha} (B)},$$
where
$$\eta_2:=\frac{N'+2 \tau+4}{2}>0.$$
Therefore, our claim \eqref{2.4} holds.

Now fix $\sigma>0$, we claim that there exists a subsequence $\{{\tilde v}_{m_i} \}_{i=1}^\infty$
of $\{{\tilde v}_m \}_{m=1}^\infty$ such that
\begin{equation}
\label{2.6}
\lim_{i,j \to \infty} \sup \|{\tilde v}_{m_i}-{\tilde v}_{m_j} \|_{L^1 (B)} \leq \sigma.
\end{equation}
To show this, we choose $\epsilon>0$ so small that
$$\|{\tilde v}_m^\epsilon-{\tilde v}_m \|_{L^1 (B)} \leq \frac{\sigma}{2}, \;\;\;
\mbox{for $m=1,2, \ldots$.}$$
Since $\{{\tilde v}_m^\epsilon \}_{m=1}^\infty$ is precompact in $L^1 (B)$, there exists a subsequence $\{{\tilde v}_{m_i}^\epsilon\}_{i=1}^\infty$ of $\{{\tilde v}_m^\epsilon \}_{m=1}^\infty$ that converges uniformly in $L^1 (B)$. Therefore, $\lim_{i,j \to \infty} \sup \|{\tilde v}_{m_i}^\epsilon-{\tilde v}_{m_j}^\epsilon \|_{L^1 (B)}=0$. Then
\begin{eqnarray*}
& & \int_B |{\tilde v}_{m_i}-{\tilde v}_{m_j}| dx \\
& &\;\;\;\;\;\; \leq  \int_B |{\tilde v}_{m_i}-{\tilde v}_{m_i}^\epsilon| dx +\int_B |{\tilde v}_{m_i}^\epsilon-{\tilde v}_{m_j}^\epsilon| dx +\int_B |{\tilde v}_{m_j}^\epsilon-{\tilde v}_{m_j}| dx\\
& &\;\;\;\;\;\; \leq \frac{\sigma}{2}+\frac{\sigma}{2} \;\;\; \mbox{as $i, j \to \infty$}.
\end{eqnarray*}
The claim \eqref{2.6} holds. We next employ assertion \eqref{2.6} with $\sigma=1, \frac{1}{2}, \ldots, \frac{1}{i}, \ldots$ and a standard diagonal argument to extract a subsequence $\{{\tilde v}_i \}_{i=1}^\infty$ of $\{{\tilde v}_m \}_{m=1}^\infty$ satisfying
$$\lim_{i,j \to \infty} \sup \|{\tilde v}_i-{\tilde v}_j \|_{L^1 (B)}=0$$
and hence ${\tilde v}_i$ converges to ${\tilde v} \in L^1 (B)$ as $i \to \infty$, i.e.,
\begin{equation}
\label{2.7}
\int_B ||x|^l {\tilde u}_i-{\tilde v}| dx \to 0 \;\; \mbox{as $i \to \infty$}.
\end{equation}

Let ${\tilde u}=|x|^{-l} {\tilde v}$. We easily see from \eqref{2.7} that ${\tilde u} \in L_l^1 (B)$ and
\begin{equation}
\label{2.8}
\int_B |x|^l |{\tilde u}_i-{\tilde u}| dx \to 0 \;\;\mbox{as $i \to \infty$}.
\end{equation}
This also implies that $u_i$ converges to $u$ in $L_l^1 (\Omega)$ and the proof of step 1 is complete.

{\it Step 2.} We show that if $\{u_m\}_{m=1}^\infty$ is a bounded sequence in $D_0^{2, \alpha} (\Omega)$, there exists a subsequence (still denoted by $\{u_m\}_{m=1}^\infty$) such that $u_m$ converges in $L^q_l (\Omega)$ for $q \in (1, p_s+1)$.

We see that for $q \in (1, p_s+1)$, there exists $\lambda \in (0,1)$ such that $\frac{1}{q}=\lambda+\frac{1-\lambda}{p_s+1}$. Then for any $u_m \in D_0^{2, \alpha} (\Omega)$,
\begin{eqnarray*}
\Big(\int_\Omega |x|^l |u_m|^q dx \Big)^{\frac{1}{q}} &\leq& \Big(\int_\Omega |x|^l |u_m| dx \Big)^\lambda \Big(\int_\Omega |x|^l |u_m|^{p_s+1} dx \Big)^{\frac{1-\lambda}{p_s+1}} \\
&\leq& C \Big(\int_\Omega |x|^l |u_m| dx \Big)^\lambda \Big(\int_\Omega |x|^\alpha (\Delta u_m)^2 dx \Big)^{\frac{1-\lambda}{2}}\\
&\leq&C \Big(\int_\Omega |x|^l |u_m| dx \Big)^\lambda.
\end{eqnarray*}
It follows from the proof of step 1 that there is a subsequence $\{u_i\}_{i=1}^\infty$ of $\{u_m\}_{m=1}^\infty$ such that $\|u_i-u_j\|_{L_l^1 (\Omega)} \to 0$ as $i, j \to \infty$.
Then
\begin{equation}
\label{2.9}
\Big(\int_\Omega ||x|^{\frac{l}{q}} u_i-|x|^{\frac{l}{q}} u_j|^q dx \Big)^{\frac{1}{q}} \leq
C \Big(\int_\Omega |x|^l |u_i-u_j| dx \Big)^\lambda \to 0 \;\;
\mbox{as $i,j \to \infty$}.
\end{equation}
There exists $v \in L^q (\Omega)$ such that
\begin{equation}
\label{2.10}
\int_\Omega ||x|^{\frac{l}{q}} u_i-v|^q dx \to 0 \;\; \mbox{as $i \to \infty$}.
\end{equation}
Set $u=|x|^{-\frac{l}{q}} v$. We see from \eqref{2.10} that $u \in L_l^q (\Omega)$ and $u_i$ converges to $u$ in $L^q_l (\Omega)$ as $i \to \infty$. This completes the proof of step 2 and the proof of this proposition. \qed

{\bf Proof of Proposition \ref{p1.3}}

Since $\alpha<\frac{N-4}{4} \tau$, we can choose $l_1$ such that $\tau_1:=l_1-\alpha>-4$,
$\frac{N-4}{4} \tau_1=\alpha$. This implies that $l_1<l$ and $p_s^1:=\frac{N'+4+2 \tau_1}{N'-4}=\frac{N+4}{N-4}$.
It follows from
Proposition \ref{p1.2} that the embedding:
\begin{equation}
\label{2.11}
D^{2, \alpha}_0 (\Omega) \hookrightarrow L_{l_1}^{q} (\Omega)
\end{equation}
is continuous for any $q \in [1, p_s^1+1] \; (:=[1, \frac{2N}{N-4}])$ and this embedding is continuous and compact for any $q \in [1, \frac{2N}{N-4})$. Moreover,
$$\int_\Omega |x|^l |u|^q dx=\int_\Omega |x|^{l_1} |x|^{l-l_1} |u|^q dx
\leq C \int_\Omega |x|^{l_1} |u|^q dx \;\; (\mbox{note that $l_1<l$}).$$
This implies that
\begin{equation}
\label{2.12}
D^{2, \alpha}_0 (\Omega) \hookrightarrow L_l^{q} (\Omega)
\end{equation}
for any $q \in [1, \frac{2N}{N-4}]$ is continuous and compact for any $q \in [1, \frac{2N}{N-4})$. \qed

{\bf Proof of Proposition \ref{p1.4}}

The conclusions of the special case of $\alpha=0$ and $l>0$ have been obtained in Theorem 1.1 and Corollary 1.2 of \cite{FSM}. We prove this proposition by different arguments.

For any $u \in C_c^2(B \backslash \{0\})$ and $u(x)=u(r)$ with $r=|x|$, we
make the transformation:
$$u(r)=r^{-\frac{N'-4}{2}} w(t), \;\;\; t=-\ln r.$$
Then, we have $w \in H^2_0 (-\ln R, \infty)$ and
\begin{equation}
\label{2.13}
\int_B |x|^\alpha (\Delta u)^2 dx=\omega_N \int_{-\ln R}^\infty  \Big(|w''|^2+2 {\tilde \delta}_\alpha |w'|^2+\delta_\alpha |w|^2 \Big) dt
\end{equation}
where $\omega_N=|S^{N-1}|$,
$${\tilde \delta}_\alpha=\Big(\frac{N-2}{2} \Big)^2+\Big(\frac{\alpha-2}{2} \Big)^2, \;\;\;
\delta_\alpha=\Big[\Big(\frac{N-2}{2} \Big)^2-\Big(\frac{\alpha-2}{2} \Big)^2 \Big]^2.$$
We see that $\delta_a \neq 0$ provided $4-N<\alpha<N$ and $N \geq 5$. Moreover,
\begin{equation}
\label{2.14}
\int_B |x|^l |u|^{p_s+1} dx=\omega_N \int_{-\ln R}^\infty |w|^{p_s+1} dt.
\end{equation}
Since $H^2_0 (-\ln R, \infty) \hookrightarrow L^{p_s+1} (-\ln R, \infty)$ by the Sobolev embedding theorem, we obtain from \eqref{2.13} and \eqref{2.14} that the embedding:
$$D^{2, \alpha}_{0, \mbox{{\tiny rad}}} (B) \hookrightarrow L_l^{p_s+1} (B)$$
is continuous. The embedding:
$$D^{2, \alpha}_{0, \mbox{{\tiny rad}}} (B) \hookrightarrow L_l^{q} (B)$$
is continuous and compact for any $q \in [1, p_s+1)$ can be obtained from the interpolations arguments similar to those in the proof of Proposition \ref{p1.2}. \qed

We now obtain the existence of positive solutions of (Q). To simplify notation, $H_\alpha (\Omega)$ stands for $W^{2,2}_\alpha (\Omega) \cap D_0^{1, \alpha} (\Omega)$, where $W^{2,2}_\alpha (\Omega)$
is the completion of $C^\infty (\Omega \backslash \{0\})$
with respect to the norm
$$\|\phi\|_{W^{2,2}_\alpha (\Omega)}^2=\int_\Omega |x|^\alpha (|u|^2+|\nabla u|^2+|D^2 u|^2) dx.$$
We see that the norm
$[\int_\Omega |x|^\alpha (|u|^2+|\nabla u|^2+|D^2 u|^2) dx ]^{1/2}$ of $W^{2,2}_\alpha (\Omega)$
is equivalent to $(\int_\Omega |x|^\alpha (\Delta u)^2 dx)^{1/2}$ on $H_\alpha (\Omega)$. We consider
$H_\alpha (\Omega)$ endowed with the norm
$$\|u\|_{H_\alpha (\Omega)}=\Big(\int_\Omega |x|^\alpha (\Delta u)^2 dx \Big)^{\frac{1}{2}}, \;\;\; u \in H_\alpha (\Omega).$$

\begin{rem}
\label{r2.1}
It is known from Remarks 2.2 and 2.3 of \cite{CM} that the functions
$$u \mapsto \max_{i,j=1, \ldots, N} \int_\Omega |x|^\alpha |\partial_{i,j} u|^2 dx,
\;\;\; u \mapsto  \int_\Omega |x|^\alpha (\Delta u)^2 dx$$
define two equivalent norms in $C_c^\infty (\Omega \backslash \{0\})$ provided $4-N<\alpha<N$.
Moreover, we see from \cite{GM1, GM2, Mo} that
$$\int_\Omega |x|^\alpha |\nabla u|^2 dx \leq C \int_\Omega |x|^{\alpha-2} |\nabla u|^2 dx
\leq C \int_\Omega |x|^\alpha (\Delta u)^2 dx$$
for any $u \in D^{2, \alpha}_0 (\Omega)$. Note that $|x|^\alpha \leq C |x|^{\alpha-2}$
for $x \in \Omega$.
\end{rem}

\begin{defn}
\label{d2.1}
We say that $u \in H_\alpha (\Omega)$ is a weak solution of
\begin{equation}
\label{2.15}
\left \{ \begin{array}{ll} \Delta (|x|^\alpha \Delta u)=|x|^l |u|^{p-1} u \;\; & \mbox{in $\Omega$},\\
 u=\Delta u=0, \;\; &\mbox{on $\partial \Omega$},
 \end{array} \right.
 \end{equation}
if $u$ is a critical point of the $C^1 (H_\alpha (\Omega), \R)$ functional
$$J(u)=\frac{1}{2} \int_\Omega |x|^\alpha (\Delta u)^2 dx-\frac{1}{p+1} \int_\Omega |x|^l |u|^{p+1} dx, \;\;\; u \in H_\alpha (\Omega),$$
that is, $u \in H_\alpha (\Omega)$ and satisfies
$$ \int_\Omega |x|^\alpha \Delta u \Delta v dx=\int_\Omega |x|^l |u|^{p-1}u v dx, \;\;\;
\forall v \in H_\alpha (\Omega).$$
\end{defn}

Now, we set
$$m_{\alpha, l}=\inf_{u \in H_\alpha (\Omega), u \neq 0} \frac{\int_\Omega |x|^\alpha (\Delta u)^2 dx}{(\int_\Omega |x|^l |u|^{p+1} dx)^{2/(p+1)}}.$$
We see that under the assumptions of Proposition \ref{p1.2}, $m_{\alpha,l}$ is attained by a function $u \in H_\alpha (\Omega)$ provided $1<p<p_s$, since the embedding: $H_\alpha (\Omega)
\hookrightarrow L^{q}_l (\Omega)$ is continuous and compact for any $q \in [1, p_s+1)$. In addition, by standard arguments, a suitable multiple of $u$ turns to be a weak solution, as defined above, of \eqref{2.15}.

\begin{thm}
\label{t2.1}
Let $\Omega \subset \R^N$ be a bounded smooth domain, but not a ball, with $0 \in \Omega$. Assume also that $N \geq 5$ and \eqref{1.1} holds with $\alpha \geq \frac{(N-4)}{4} \tau$. Let $u \in H_\alpha (\Omega)$, $u \not \equiv 0$ be a minimizer of $m_{\alpha,l}$ for $1<p<p_s$.
Then $u \in C^4 (\Omega \backslash \{0\}) \cap C^0 ({\overline \Omega})$, $|x|^\alpha \Delta u \in C^0 ({\overline \Omega})$ and $u>0, -|x|^\alpha \Delta u>0$
in $\Omega$. Therefore, (Q) admits a positive solution $u$.
\end{thm}

{\bf Proof.}
Let $u \in H_\alpha (\Omega)$ be a minimizer for $m_{\alpha,l}$.
The conclusions of this theorem are two parts:

(i) $u>0$ and $-|x|^\alpha \Delta u>0$ in $\Omega$,

(ii) $u \in C^4 (\Omega \backslash \{0\}) \cap C^0 ({\overline \Omega})$, $|x|^\alpha \Delta u \in C^0 ({\overline \Omega})$.\\
We prove the first part here. The proof of the second part will be obtained in the next section. In fact the proof of the second part relies on the assumptions $u>0$ and $-|x|^\alpha \Delta u>0$
in $\Omega \backslash \{0\}$.

To see the conclusions in (i),
it suffices to prove that $\Delta u$ does not change sign in $\Omega$. Indeed, suppose by contradiction that it does. Consider $w$ be the solution of
$$-\Delta w=|\Delta u| \;\;\mbox{in $\Omega$}, \;\;\; w=0 \;\; \mbox{on $\partial \Omega$}.$$
Since
$$\int_\Omega |x|^\alpha (\Delta w)^2 dx=\int_\Omega |x|^\alpha (\Delta u)^2 dx,$$
we see that $w \in H_\alpha (\Omega)$. Regularity of $-\Delta$ and $\Delta^2$ and the boot-strap argument imply that $u \in C^2(\Omega \backslash \{0\})$ and $w \in C^2 (\Omega \backslash \{0\})$.
(Note that $p_s \leq \frac{N+4}{N-4}$, for any $\Omega' \subset \subset \Omega \backslash \{0\}$, we can show $u \in C^4 (\Omega')$ by the boot-strap argument.) Observe that $-\Delta (w \pm u) \geq 0$ in $\Omega$. Then, the strong maximum principle implies $w>|u|$ in $\Omega \backslash \{0\}$. Using $w$ in the quotient that defines $m_{\alpha,l}$ we get a contradiction.
Since $(u, v)$ satisfies the following system of equations:
\begin{equation}
\label{2.16}
\left \{ \begin{array}{ll} -\Delta u=|x|^{-\alpha} v \;\; &\mbox{in $\Omega$},\\
-\Delta v=|x|^l |u|^{p-1}u \;\; &\mbox{in $\Omega$},\\
u=v=0 \;\; &\mbox{on $\partial \Omega$},
\end{array} \right.
\end{equation}
regularity of $-\Delta$ and the boot-strap argument imply that $u \in C^2 (\Omega \backslash \{0\})$ and
$v \in C^2 (\Omega \backslash \{0\})$.
The strong maximum principle guarantees $u$, $v>0$ in $\Omega \backslash \{0\}$ or $u$, $v<0$ in $\Omega \backslash \{0\}$. Without loss of generality, we assume the first case occurs. The proof in the next section implies that the limits $\lim_{|x| \to 0} u(x)$ and $\lim_{|x| \to 0} v(x)$ exist and to be positive. Therefore, the conclusions in (i) hold. If the second case occurs, we easily see that $-u$, $-v>0$ in $\Omega$. In the next section, we will see that $u \in C^4 (\Omega \backslash \{0\}) \cap C^0 ({\overline \Omega})$, $|x|^\alpha \Delta u \in C^0 ({\overline \Omega})$ provided $u>0$, $-|x|^\alpha \Delta u>0$ in $\Omega \backslash \{0\}$. The proof of this theorem is complete. \qed

Similarly, we have the following theorems.

\begin{thm}
\label{t2.2}
Let $\Omega \subset \R^N$ be a bounded smooth domain, but not a ball, with $0 \in \Omega$. Assume also that $N \geq 5$ and \eqref{1.1} holds with $\alpha<\frac{(N-4)}{4} \tau$. Let $u \in H_\alpha (\Omega)$, $u \not \equiv 0$ be a minimizer of $m_{\alpha,l}$ for $1<p<\frac{N+4}{N-4}$.
Then $u \in C^4 (\Omega \backslash \{0\}) \cap C^0 ({\overline \Omega})$, $|x|^\alpha \Delta u \in C^0({\overline \Omega})$ and $u>0$, $-|x|^\alpha \Delta u>0$
in $\Omega$. Therefore, (Q) admits a positive solution $u$.
\end{thm}

\begin{thm}
\label{t2.3}
Let $B$ be a ball centered at 0. Assume also that $N \geq 5$ and \eqref{1.1} holds. Let $u \in H_{\alpha, \mbox{{\tiny rad}}} (B)$, $u \not \equiv 0$ be a minimizer of $m_{\alpha,l}$ for $1<p<p_s$.
Then
$u \in C^4 (B \backslash \{0\}) \cap C^0 ({\overline B})$, $|x|^\alpha \Delta u \in C^0({\overline B})$ and $u>0$, $-|x|^\alpha \Delta u>0$
in $B$. Therefore, (Q) admits a positive radial solution $u$.
\end{thm}

\section{Removable singularities of solutions of (P)}
\setcounter{equation}{0}

In this section, we obtain some removable singularity results for solutions to the equation  (P).
Such results can be seen as regularity results of solutions to the equation  (P), which have important applications in the boundary value problems.

\begin{lem}
\label{l3.1}
Assume that $1<p<p_s$ and \eqref{1.1} holds. Let $u(x)=u(|x|) \in H_{\alpha,\mbox{{\tiny rad}}} (B)$ be a weak radial solution to the problem (Q) and $u>0$, $-|x|^\alpha \Delta u>0$
in $B \backslash \{0\}$.
Then $u \in C^4 (B \backslash \{0\}) \cap C^0 ({\overline B})$, $|x|^\alpha \Delta u \in C^0 ({\overline B})$.
\end{lem}

{\bf Proof.} Define $v(r)=-r^\alpha \Delta u(r)$. We see that $(u, v)$ satisfies the system of equations
\begin{equation}
\label{3.1}
\left \{ \begin{array}{ll} -(r^{N-1} u' (r))'=r^{N-\alpha-1} v (r) \;\; &\mbox{in $(0,R)$},\\
-(r^{N-1} v'(r))'=r^{N'+\tau-1} u^p(r) \;\; &\mbox{in $(0,R)$},\\
u (R)=v (R)=0.
\end{array} \right.
\end{equation}
Moreover, we easily see from the equations in \eqref{3.1} that
$r^{N-1} u'(r)$ and $r^{N-1} v'(r)$ are decreasing functions for $r \in (0,R)$ and thus both
$\lim_{r \to 0} r^{N-1} u'(r)$ and $\lim_{r \to 0} r^{N-1} v'(r)$ exist (maybe $\infty$). The facts $u \in H_{\alpha,\mbox{{\tiny rad}}} (B)$ and the embedding $H_{\alpha,\mbox{{\tiny rad}}} (B) \hookrightarrow L^q_l (B)$
for $1 \leq q \leq p_s+1$ imply that
\begin{equation}
\label{3.2}
\int_B |x|^l u^p dx<\infty, \;\;\; \int_B |x|^{-\alpha} v^2 dx<\infty,
\end{equation}
and
\begin{equation}
\label{3.3}
\int_B |x|^{-\alpha} |v| dx \leq \Big(\int_B |x|^{-\alpha} v^2 dx \Big)^{\frac{1}{2}} \Big(\int_B |x|^{-\alpha} dx \Big)^{\frac{1}{2}}<\infty,
\end{equation}
since $\alpha<N$. Therefore, for any $\epsilon>0$,
$$-\int_{B_\epsilon} |x|^l u^p dx=\int_{B_\epsilon} \Delta v dx
=C \epsilon^{N-1} v'(\epsilon),$$
$$-\int_{B_\epsilon} |x|^{-\alpha} v dx=\int_{B_\epsilon} \Delta u dx
=C \epsilon^{N-1} u'(\epsilon).$$
These, \eqref{3.2} and \eqref{3.3} imply
that $r^{N-1} u'(r) \to 0$, $r^{N-1} v'(r) \to 0$ as $r \to 0$. Integrations of the equations in \eqref{3.1} give that $u \in C^2 (0, R)$, $v \in C^2 (0,R)$.

To prove this lemma, we only need to show the following claim:
\begin{equation}
\label{3.4}
\lim_{r \to 0} u(r)=\tau_u \in (0, \infty), \;\;\; \lim_{r \to 0}  v(r)=\tau_v \in (0, \infty),
\end{equation}
where $\tau_u$ and $\tau_v$ are constants.

Since $u \in H_{\alpha,\mbox{{\tiny rad}}} (B)$ and \eqref{1.1} holds, we have
\begin{equation}
\label{3.4-1}
\int_0^R r^{N'-1-2 \alpha} v^2 (r) dr<\infty,
\end{equation}
note that $\Delta u=-|x|^{-\alpha} v$.
We also obtain from the embedding:
$$H_{\alpha,\mbox{{\tiny rad}}} (B) \hookrightarrow L^{p_s+1}_l (B)$$
that
\begin{equation}
\label{3.5}
\int_0^R r^{N'+\tau-1} u^{\frac{2(N'+\tau)}{N'-4}} (r) dr<\infty.
\end{equation}
This implies that, for $r$ near 0,
\begin{equation}
\label{3.6}
u(r)=o \Big(r^{-\frac{N'-4}{2}} \Big), \;\;\; v (r)=o \Big(r^{-\frac{N-\alpha}{2}} \Big).
\end{equation}
We only show $\eqref{3.6}_1$ by using \eqref{3.5}. The proof of $\eqref{3.6}_2$ is similar to that of $\eqref{3.6}_1$ by using \eqref{3.4-1}. We easily see from \eqref{3.5} that
$$\int_0^r s^{N'+\tau-1} u^{\frac{2(N'+\tau)}{N'-4}} (s) ds=o(1) \;\; \mbox{for $r$ near 0}.$$
Using the fact that $u(r)$ is decreasing, we have that
$$\int_0^r s^{N'+\tau-1} u^{\frac{2(N'+\tau)}{N'-4}} (s) ds \geq \frac{r^{N'+\tau}}{N'+\tau}
u^{\frac{2(N'+\tau)}{N'-4}} (r).$$
This implies that
$$r^{N'+\tau} u^{\frac{2(N'+\tau)}{N'-4}} (r)=o(1)$$
and
$$u(r)=o \Big(r^{-\frac{N'-4}{2}} \Big).$$

Let
$$w(t)=r^{\frac{N'-4}{2}} u(r), \;\; z(t)=r^{\frac{N-\alpha}{2}} v (r), \;\;\; t=-\ln r.$$
We see that $(w(t), z(t))$ satisfies the problem
\begin{equation}
\label{3.7}
\left \{ \begin{array}{ll} w_{tt}+(\alpha-2) w_t-\frac{(N-\alpha) (N'-4)}{4} w +z=0, \;\; & t \in (-\ln R, \infty),\\
z_{tt}+(2-\alpha) z_t-\frac{(N-\alpha) (N'-4)}{4} z+e^{-p_* t} w^p=0, \;\; & t \in (-\ln R, \infty),\\
w(-\ln R)=z (-\ln R)=0, \end{array} \right.
\end{equation}
where and in the following
$$p_*=\frac{[N'+4+2 \tau]-[N'-4]p}{2}.$$
Note that $p_*>0$ provided $1<p<p_s$.
We know from \eqref{3.6} and $1<p<p_s$ that $w(t) \to 0$, $z(t) \to 0$ and $e^{-p_* t} \to 0$ as $t \to \infty$.
It follows from the ODE theory on perturbation of linear systems (see \cite{Ha})
that there exist $\tau_u>0$, $\tau_v>0$ such that
\begin{equation}
\label{3.8}
w(t)=(\tau_u+o(1)) e^{-\frac{N'-4}{2} t},  \;\;\; z(t)=(\tau_v+o(1)) e^{-\frac{N-\alpha}{2} t} \;\; \mbox{as $t \to \infty$}.
\end{equation}
Note that the system \eqref{3.7} can be written to the following system
\begin{equation}
\label{3.8-a}
\left \{ \begin{array}{l}
\dot{w_1}=w_2,\\
\dot{w_2}=\frac{(N-\alpha)(N'-4)}{4} w_1-(\alpha-2) w_2-w_3,\\
\dot{w_3}=w_4,\\
\dot{w_4}=\frac{(N-\alpha)(N'-4)}{4} w_3+(\alpha-2) w_4-e^{-p_*t}w_1^p,
\end{array} \right.
\end{equation}
where $(w_1,w_2,w_3,w_4)=(w, \dot{w},z,\dot{z})$. Since $(w_1 (t),w_2 (t),w_3 (t),w_4 (t)) \to (0,0,0,0)$ as $t \to \infty$, we linearize the system \eqref{3.8-a} at (0,0,0,0) and obtain the system
\begin{equation}
\label{3.8-b}
\left \{ \begin{array}{l}
\dot{h_1}=h_2,\\
\dot{h_2}=\frac{(N-\alpha)(N'-4)}{4} h_1-(\alpha-2) h_2-h_3,\\
\dot{h_3}=h_4,\\
\dot{h_4}=\frac{(N-\alpha)(N'-4)}{4} h_3+(\alpha-2) h_4.
\end{array} \right.
\end{equation}
The matrix of \eqref{3.8-b} is
\begin{equation}
\label{3.8-c}
{\mathcal A}=\left ( \begin{array}{llll} 0 &1 &0 &0\\
\frac{(N-\alpha)(N'-4)}{4} &2-\alpha &-1 &0\\
0 &0 &0 &1\\
0 &0 &\frac{(N-\alpha)(N'-4)}{4} &\alpha-2
\end{array} \right ).
\end{equation}
By simple calculations, we see that the four eigenvalues of $\mathcal{A}$ are: $\lambda_{1,2}=\mp \frac{N'-4}{2}$, $\lambda_{3,4}=\mp \frac{N-\alpha}{2}$. We only choose $\lambda_1=-\frac{N'-4}{2}$
and $\lambda_3=-\frac{N-\alpha}{2}$, since $\lambda_2=\frac{N'-4}{2}>0$ and $\lambda_4=\frac{N-\alpha}{2}>0$ which do not meet our requirement. We also see that
there are eigenvectors for $\lambda_1$ and $\lambda_2$ respectively: $e^{\lambda_1t}(1, -\frac{N'-4}{2}, 0, 0)$
and $e^{\lambda_2 t}(0,0, 1, -\frac{N-\alpha}{2})$.
These imply that \eqref{3.8} holds and hence our claim \eqref{3.4} holds. The proof of this lemma
is complete. \qed

\begin{lem}
\label{l3.2}
Let $\Omega \subset \R^N$ be a bounded smooth domain, but not a ball, with $0 \in \Omega$. Assume also that $N \geq 5$ and \eqref{1.1} holds with $\alpha \geq \frac{(N-4)}{4} \tau$. Let $u \in H_\alpha (\Omega)$ be a weak solution of (Q) for $1<p<p_s$ and $u>0$, $-|x|^\alpha \Delta u>0$
in $\Omega \backslash \{0\}$.
Then
$u \in C^4 (\Omega \backslash \{0\}) \cap C^0 ({\overline \Omega})$, $|x|^\alpha \Delta u \in C^0 ({\overline \Omega})$.
\end{lem}

{\bf Proof.} We know that $p_s \leq \frac{N+4}{N-4}$ provided that \eqref{1.1} holds with
$\alpha \geq \frac{N-4}{4} \tau$.
Then the facts $u \in H_\alpha (\Omega)$ and $1<p<p_s$ and the boot-strap argument imply that
$u \in C^4 (\Omega \backslash \{0\})$. Note that for any subset $\Omega' \subset \Omega \backslash \{0\}$, $u \in W^{2,2} (\Omega')$. The boot-strap argument (note that $1<p<\frac{N+4}{N-4}$ if $1<p<p_s$) implies that
$u \in C^4 (\Omega')$.

Since $0 \in \Omega$, we can choose a $0<\sigma<\frac{1}{10}$ such that $B_\sigma \subset \subset \Omega$, where $B_\sigma=\{x \in \Omega: \; |x|<\sigma\}$. The proof is divided to 2 steps:

(1) We show that there exists $C>0$ such that
\begin{equation}
\label{3.8-1}
u(x) \leq C |x|^{-\frac{4+\tau}{p-1}}, \;\;\; v(x) \leq C |x|^{\alpha-\frac{2(p+1)+\tau}{p-1}},
\;\;\; 0<|x|<\sigma/2,
\end{equation}
where $v(x)=-|x|^\alpha \Delta u (x)$. Note that $u>0$, $v>0$ in $B_{\sigma/2}$.

(2) Using the estimates obtained in \eqref{3.8-1} and Harnack's inequality of systems,
we obtain our conclusion.

{\it Step 1.} For any $x_0 \in B_\sigma$ with $0<|x_0|<\sigma/2$, we denote
$$R=\frac{|x_0|}{2}$$
and observe that, for all $y \in B_1$, $\frac{|x_0|}{2}<|x_0+Ry|<\frac{3 |x_0|}{2}$, so that
$x_0+R y \in B_\sigma$. Let us define
$$U(y)=R^{\frac{4+\tau}{p-1}} u(x_0+Ry), \;\;\; V(y)=R^{\frac{2(p+1)+\tau}{p-1}-\alpha} v(x_0+Ry).$$
Then $(U,V)$ satisfies the system
\begin{equation}
\label{3.8-2}
\left \{ \begin{array}{ll} -\Delta U=|y+\frac{x_0}{R}|^{-\alpha} V \;\; &\mbox{in $B_1$},\\
-\Delta V=|y+\frac{x_0}{R}|^l U^p \;\; &\mbox{in $B_1$}.
\end{array} \right.
\end{equation}
Note that $|y+\frac{x_0}{R}| \in [1,3]$ for all $y \in {\overline {B_1}}$.

We now claim that there exists $C>0$ depending only on $p, N', \tau$, independent of $x_0$, such that
\begin{equation}
\label{3.8-3}
U^{\frac{p-1}{4}} (y)+V^{\frac{p-1}{2 (p+1)}} (y) \leq C (1+\mbox{dist}^{-1} (y, \partial B_1)), \;\; y \in B_1.
\end{equation}
Arguing by contradiction, we suppose that there exist sequences $x_k \in B_{\sigma/2} \backslash \{0\}$, points $z_k \in B_1$ and $(U_k,V_k)$
satisfying
\begin{equation}
\label{3.8-4}
\left \{ \begin{array}{ll} -\Delta U_k=|y+\frac{x_k}{R_k}|^{-\alpha} V_k \;\; &\mbox{in $B_1$},\\
-\Delta V_k=|y+\frac{x_k}{R_k}|^l U_k^p \;\; &\mbox{in $B_1$},
\end{array} \right.
\end{equation}
where $R_k=\frac{1}{2}|x_k|$ such that
the functions
$$M_k=|U_k|^{\frac{p-1}{4}}+|V_k|^{\frac{p-1}{2(p+1)}}$$
satisfy
$$M_k (z_k)>2k (1+\mbox{dist}^{-1} (z_k, \partial B_1))>2k \mbox{dist}^{-1} (z_k, \partial B_1).$$
By  doubling lemma in \cite{PQS}, there exists $y_k \in B_1$ such that
$$M_k (y_k) \geq M_k (z_k), \;\; M_k (y_k)>2k \mbox{dist}^{-1} (y_k, \partial B_1),$$
and
\begin{equation}
\label{3.8-5}
M_k (y) \leq 2 M_k (y_k), \;\; \mbox{for all $y$ such that $|y-y_k| \leq k M_k^{-1} (y_k)$}.
\end{equation}
We have
\begin{equation}
\label{3.8-6}
\lambda_k:=M_k^{-1} (y_k) \to 0\;\; \mbox{as $k \to \infty$},
\end{equation}
due to $M_k (y_k) \geq M_k (z_k)>2k$.

Next we let
$${\tilde U}_k (z)=\lambda_k^{\frac{4}{p-1}} U_k (y_k+\lambda_k z), \;\;
{\tilde V}_k (z)=\lambda_k^{\frac{2(p+1)}{p-1}} V_k (y_k+\lambda_k z).$$
Note that $|{\tilde U}_k|^{\frac{p-1}{4}} (0)+|{\tilde V}_k|^{\frac{p-1}{2(p+1)}} (0)=1$,
\begin{equation}
\label{3.8-7}
\Big[|{\tilde U}_k|^{\frac{p-1}{4}}+|{\tilde V}_k|^{\frac{p-1}{2(p+1)}} \Big] (z) \leq 2, \;\; |z| \leq k,
\end{equation}
due to \eqref{3.8-5} and we see that $({\tilde U}_k, {\tilde V}_k)$ satisfies
\begin{equation}
\label{3.8-8}
\left \{ \begin{array}{ll} -\Delta {\tilde U}_k=|y_k+\lambda_k z+\frac{x_k}{R_k}|^{-\alpha} {\tilde V}_k, \;\; &|z| \leq k,\\
-\Delta {\tilde V}_k=|y_k+\lambda_k z+\frac{x_k}{R_k}|^l {\tilde U}_k^p, \;\; &|z| \leq k.
\end{array} \right.
\end{equation}

Now, for each $A>0$ and $1<q<\infty$, by \eqref{3.8-7}, \eqref{3.8-8} and interior elliptic $L^q$ estimates, the sequence $({\tilde U}_k, {\tilde V}_k)$ is uniformly bounded in $W^{2, q} (B_A) \times W^{2, q} (B_A)$. Note that
$$\Big|y_k+\lambda_k z+\frac{x_k}{R_k} \Big| \in [1,3] \;\; \mbox{for all $k$}.$$
Using standard embeddings and interior elliptic Schauder estimates, after extracting subsequences
(still denoted by $\{y_k\}$, $\{\frac{x_k}{R_k}\}$ and $\{({\tilde U}_k, {\tilde V}_k)\}$), we may assume $y_k \to y_0 \in {\overline {B_1}}$, $\frac{x_k}{R_k} \to {\tilde x} \in \partial B_2$,
$({\tilde U}_k, {\tilde V}_k) \to ({\tilde U}, {\tilde V})$ in $C^2_{loc} (\R^N) \times C^2_{loc} (\R^N)$. It follows that $({\tilde U}, {\tilde V}) \geq 0$ is a classical solution of
\begin{equation}
\label{3.8-9}
\left \{ \begin{array}{ll} -\Delta {\tilde U}=|y_0+{\tilde x}|^{-\alpha} {\tilde V}, \;\;
&z \in \R^N,\\
-\Delta {\tilde V}=|y_0+{\tilde x}|^l {\tilde U}^p, \;\; & z \in \R^N.
\end{array} \right.
\end{equation}
\eqref{3.8-9} implies that ${\tilde U} \in C^4 (\R^N \backslash \{0\}) \cap C^2 (\R^N)$ satisfies
\begin{equation}
\label{3.8-10}
\Delta^2 {\tilde U}=C {\tilde U}^p \;\; \mbox{in $\R^N \backslash \{0\}$},
\end{equation}
where $0<C=|y_0+{\tilde x}|^\tau<\infty$ and ${\tilde U}^{\frac{p-1}{4}} (0)+{\tilde V}^{\frac{(p-1)}{2(p+1)}} (0)=1$. Since $1<p<p_s \leq \frac{N+4}{N-4}$, this contradicts the Liouville-type result in \cite{Lin} and concludes our claim \eqref{3.8-3}. It also implies that
$U(0)+V(0) \leq C$, hence
$$u(x_0) \leq C |x_0|^{-\frac{4+\tau}{p-1}}, \;\; v(x_0) \leq C |x_0|^{\alpha-\frac{2(p+1)+\tau}{p-1}},$$
which completes the proof of step 1.

{\it Step 2.} Let $\sigma_0=\sigma/4$. We make the transform:
\begin{equation}
\label{3.11}
w (t, \omega)=r^{\frac{N'-4}{2}} u(r \omega), \;\; z(t, \omega)=r^{\frac{N-\alpha}{2}} v(r \omega), \;\;\;  t=-\ln r.
\end{equation}
Note that $u>0$ and $v>0$ in $B_{\sigma_0} \subset \subset \Omega$.
Then $(w(t, \omega), z(t, \omega))$ satisfies
\begin{equation}
\label{3.12}
\left \{ \begin{array}{ll} w_{tt}+\Delta_{S^{N-1}} w+(\alpha-2) w_t-\frac{(N-\alpha) (N'-4)}{4} w +z=0, \;\; & (t, \omega) \in I \times S^{N-1},\\
z_{tt}+\Delta_{S^{N-1}} z+(2-\alpha) z_t-\frac{(N-\alpha) (N'-4)}{4} z+e^{-p_* t} w^p=0, \;\; &
(t, \omega) \in I \times S^{N-1},
\end{array} \right.
\end{equation}
where $I=(-\ln \sigma_0, \infty)$ since $B_{\sigma_0} \subset \subset \Omega$.
We can write \eqref{3.12} in the form:
\begin{equation}
\label{3.13}
\left \{ \begin{array}{ll} w_{tt}+\Delta_{S^{N-1}} w+(\alpha-2) w_t-\frac{(N-\alpha) (N'-4)}{4} w +z=0, \;\; & (t, \omega) \in I \times S^{N-1},\\
z_{tt}+\Delta_{S^{N-1}} z+(2-\alpha) z_t-\frac{(N-\alpha) (N'-4)}{4} z+a(t, \omega) w=0, \;\; &
(t, \omega) \in I \times S^{N-1},
\end{array} \right.
\end{equation}
with $a(t,\omega)=e^{-p_* t} w^{p-1}$.
It is known from \eqref{3.8-1} that
$$w(t, \omega) \leq C e^{(\frac{4+\tau}{p-1}-\frac{N'-4}{2})t} \;\;\; \mbox{for $(t, \omega)
\in I \times S^{N-1}$}.$$
Therefore,
$$e^{-p_*t} w^{p-1} \leq C e^{[(\frac{4+\tau}{p-1}-\frac{N'-4}{2})(p-1)-p_*]t}=C \;\;\;\mbox{for $(t, \omega)
\in I \times S^{N-1}$}.$$
Note that
$$\Big(\frac{4+\tau}{p-1}-\frac{N'-4}{2} \Big)(p-1)-p_*=0.$$
Since $|a(t, \omega)| \leq C$ for $t \in (-\ln \sigma_0, \infty) \times S^{N-1}$,
by  Harnack's inequality (see Theorem 2.1 of \cite{AGM}) of systems over
$[L-1,L] \times S^{N-1}$ with $L-1>-\ln \sigma_0$, there exists $C>0$ independent of $t \in [-\ln \sigma_0+1, \infty)$ such that
\begin{equation}
\label{3.14}
\max_{S^{N-1}} \Big \{w (t, \cdot), z(t, \cdot) \Big\} \leq C \min_{S^{N-1}} \Big \{w(t, \cdot), z(t, \cdot) \Big \} \;\;\; \mbox{for $t \geq -\ln \sigma_0+1$}.
\end{equation}

Let ${\overline w} (t)=\frac{1}{|S^{N-1}|} \int_{S^{N-1}} w(t, \omega) d \omega$ and
${\overline z} (t)=\frac{1}{|S^{N-1}|} \int_{S^{N-1}} z(t, \omega) d \omega$.
It follows from \eqref{3.14} that there exist $0<c_1<c_2$ and $0<c_3<c_4$ such that
$$c_1 {\overline w} \leq w \leq c_2 {\overline w} \;\; \mbox{for $(t, \omega) \in (-\ln \sigma_0+1, \infty) \times S^{N-1}$}, $$
$$c_3 {\overline z} \leq z \leq c_4 {\overline z} \;\; \mbox{for $(t, \omega) \in (-\ln \sigma_0+1, \infty) \times S^{N-1}$}. $$
These imply that
\begin{equation}
\label{3.15}
c_1 {\overline u} \leq u \leq c_2 {\overline u}, \;\;\;c_3 {\overline v} \leq v \leq c_4 {\overline v}  \;\;\; \mbox{for $x \in B_{e^{-1} \sigma_0} \backslash \{0\}$},
\end{equation}
where ${\overline u} (r):=\frac{\int_{\partial B_r} u (r, \omega) d \omega}{|\partial B_r|}$,
${\overline v} (r):=\frac{\int_{\partial B_r} v(r, \omega) d \omega}{|\partial B_r|}$.

On the other hand, it follows from $u \in H_\alpha (\Omega)$ and the embedding given in Proposition \ref{p1.2} that
\begin{equation}
\label{3.15-1}
\int_{B_{e^{-1} \sigma_0}} |x|^l |u|^{\frac{2(N'+\tau)}{N'-4}} dx<\infty, \;\;\; \int_{B_{e^{-1} \sigma_0}} |x|^{-\alpha} v^2 dx<\infty
\end{equation}
provided that \eqref{1.1} holds with $\alpha \geq \frac{N-4}{4} \tau$.
This and \eqref{3.15} imply
\begin{equation}
\label{3.16}
\int_{B_{e^{-1} \sigma_0}} |x|^l {\overline u}^{\frac{2(N'+\tau)}{N'-4}} dx<\infty,
\;\;\; \int_{B_{e^{-1} \sigma_0}} |x|^{-\alpha} {\overline v}^2 dx<\infty.
\end{equation}
Then, arguments similar to those in the proof of \eqref{3.6} imply that there exists $0<\sigma_1<e^{-1} \sigma_0$ such that
\begin{equation}
\label{3.17}
{\overline u} (r)=o \Big(r^{-\frac{N'-4}{2}} \Big), \;\;\; {\overline v} (r)=o \Big(r^{-\frac{N-\alpha}{2}} \Big) \;\;\; \forall r \in (0, \sigma_1).
\end{equation}
Note that
$${\overline w} (t)=r^{\frac{N'-4}{2}} {\overline u}(r), \;\;\;{\overline z} (t)=r^{\frac{N-\alpha}{2}} {\overline v}(r), \;\;\;  t=-\ln r.$$
\eqref{3.17} implies that
\begin{equation}
\label{3.18}
{\overline w}(t) \to 0, \;\;\; {\overline z}(t) \to 0 \;\;\;  \mbox{as $t \to \infty$}.
\end{equation}
On the other hand,
$({\overline w}, {\overline z})$ satisfies
\begin{equation}
\label{3.19}
\left \{ \begin{array}{ll} {\overline w}_{tt}+(\alpha-2) {\overline w}_t-\frac{(N-\alpha) (N'-4)}{4} {\overline w} +{\overline z}=0, \;\; & t \in I,\\
{\overline z}_{tt}+(2-\alpha) {\overline z}_t-\frac{(N-\alpha) (N'-4)}{4} {\overline z}+e^{-p_* t} {\overline {w^p}}=0, \;\; &
t \in I,
\end{array} \right.
\end{equation}
and due to \eqref{3.14},
$$c_5 {\overline w}^p \leq {\overline {w^p}} \leq c_6 {\overline w}^p \;\;
\mbox{for $(t, \omega) \in (-\ln \sigma_0+1, \infty) \times S^{N-1}$}.$$
By \eqref{3.18}, it follows from the ODE theory on perturbation of linear systems (see \cite{Ha})
that, there exist ${\tilde \tau}_1>0$ and ${\tilde \tau}_2>0$ such that
\begin{equation}
\label{3.20}
{\overline w}(t)=({\tilde \tau}_1+o(1)) e^{-\frac{N'-4}{2} t} \;\;\;{\overline z}(t)=({\tilde \tau}_2+o(1)) e^{-\frac{N-\alpha}{2} t} \;\;\;  \mbox{as $t \to \infty$}.
\end{equation}
In view of \eqref{3.14}, \eqref{3.20} implies
$$w(t, \omega)=O(e^{-\frac{N'-4}{2} t}), \;\;\; z(t, \omega)=O( e^{-\frac{N-\alpha}{2} t}) \;\;  \mbox{as $t \to \infty$ uniformly for $\omega \in S^{N-1}$}.$$
These imply that there exist $\tau_1>0$, $\tau_2>0$ such that
\begin{equation}
\label{3.21}
u(x) \to \tau_1, \;\;\; v(x) \to \tau_2 \;\; \mbox{as $|x| \to 0$}.
\end{equation}
The proof of this lemma is completed. \qed

\begin{lem}
\label{l3.3}
Let $\Omega \subset \R^N$ be a bounded smooth domain, but not a ball, with $0 \in \Omega$. Assume also that $N \geq 5$ and \eqref{1.1} holds with $\alpha<\frac{(N-4)}{4} \tau$. Let $u \in H_\alpha (\Omega)$ be a weak solution of (Q) for $1<p<\frac{N+4}{N-4}$ and $u>0$, $-|x|^\alpha \Delta u>0$ in $\Omega \backslash \{0\}$.
Then
$u \in C^4 (\Omega \backslash \{0\}) \cap C^0 ({\overline \Omega})$, $|x|^\alpha \Delta u \in C^0 ({\overline \Omega})$.
\end{lem}

{\bf Proof.} Since $\alpha<\frac{N-4}{2} \tau$, we see that $\frac{N+4}{N-4}<p_s$.
The conclusions can be obtained by arguments exactly the same as those in the proof of Lemma \ref{l3.2}.
The only thing we need to show is
\begin{equation}
\label{3.21-1}
{\overline u} (r)=o \Big(r^{-\frac{N'-4}{2}} \Big) \;\; \mbox{for $r$ near 0}.
\end{equation}
To see \eqref{3.21-1}, we choose $l_1$ and $\tau_1$ as in the proof of
Proposition \ref{p1.3},
then
$$\int_{B_\sigma} |x|^{l_1} u^{\frac{2(N'+\tau_1)}{N'-4}} dx \; \Big(:=\int_{B_\sigma} |x|^{l_1} u^{\frac{2N}{N-4}} dx \Big)  <\infty.$$
By Harnack's inequality as in the proof of Lemma \ref{l3.2}, we have
$$\int_{B_\sigma} |x|^{l_1} {\overline u}^{\frac{2(N'+\tau_1)}{N'-4}} dx<\infty.$$
Now, \eqref{3.21-1} can be obtained by similar arguments  to those in the proof of \eqref{3.6}.
This completes the proof of this lemma.
\qed

\begin{cor}
\label{c3.4}
Let $\Omega \subset \R^N$ be a bounded smooth domain with $0 \in \Omega$. Assume that \eqref{1.1} holds and $\varphi_1 \in H_\alpha (\Omega)$ is the first eigenfunction corresponding to
the first eigenvalue $\lambda_1$ of the problem
\begin{equation}
\label{3.22}
\left \{ \begin{array}{rl} \Delta (|x|^\alpha \Delta \phi (x))=\lambda |x|^l \phi (x), \;\;& x \in \Omega,\\
\phi (x)=0, \;\;\; \Delta \phi (x)=0, \;\; &x \in \partial \Omega.
\end{array} \right.
\end{equation}
Then $\varphi_1 \in C^4 (\Omega \backslash \{0\}) \cap C^0 ({\overline \Omega})$, $|x|^\alpha \Delta \varphi_1 \in C^0({\overline \Omega})$.
\end{cor}

{\bf Proof.}  The existence of a nontrivial $\varphi_1 \in
H_\alpha (\Omega)$ as a minimizer of
$$\inf_{\phi \in H_\alpha (\Omega) \backslash \{0\}}
\frac{\int_\Omega |x|^\alpha (\Delta \phi)^2}{\int_\Omega |x|^l |\phi|^2 dx}$$
can be obtained from the compact embedding: $H_\alpha \hookrightarrow L_l^2 (\Omega)$. We need to consider
two cases respectively: \eqref{1.1} holds with $\alpha \geq \frac{N-4}{4} \tau$ and \eqref{1.1} holds with $\alpha<\frac{N-4}{4} \tau$. Similar arguments  to those in the proof of Theorems \ref{t2.1} and \ref{t2.2} imply that $\varphi_1>0$, $-|x|^\alpha \Delta \varphi_1>0$ in $\Omega$.
Other properties can be also obtained. Note that we can use Harnack's inequality here, since in our case here
$p_*=4+\tau>0$ (note that $p=1$) and
$$a(t, \omega)=e^{-(4+\tau)t}, \;\;\; (t, \omega) \in I \times S^{N-1},$$
which implies that $|a(t, \omega)| \leq C$ for $(t, \omega) \in I \times S^{N-1}$.
\qed

\begin{cor}
\label{c3.5}
Let $B=\{x \in \R^N: \; |x|<R\}$ be a ball in $\R^N$ centered at 0. Assume that \eqref{1.1} holds. Then
the first eigenfunction $\varphi_1$ corresponding to the first eigenvalue $\lambda_1$ of the
eigenvalue problem
\begin{equation}
\label{3.23}
\left \{ \begin{array}{rl} \Delta (|x|^\alpha \Delta \phi (x))=\lambda |x|^l \phi (x), \;\;& x \in B,\\
\phi (x)=0, \;\; \Delta \phi (x)=0, \;\; &x \in \partial B
\end{array} \right.
\end{equation}
satisfies that $\varphi_1 \in C^4 (B \backslash \{0\}) \cap C^0 ({\overline B})$, $|x|^\alpha \Delta \varphi_1 \in C^0 ({\overline B})$,
$\varphi_1 (x)=\varphi_1 (r)$, $\varphi_1 (r)>0$ and $\varphi_1'(r)<0$ for $r \in (0, R)$.
\end{cor}

{\bf Proof.} The existence of a nontrivial nonnegative $\varphi_1 \in
H_{\alpha,\mbox{{\tiny rad}}} (B)$ as a minimizer of
$$\inf_{\phi \in H_{\alpha,\mbox{{\tiny rad}}} (B) \backslash \{0\}}
\frac{\int_B |x|^\alpha (\Delta \phi)^2}{\int_B |x|^l |\phi|^2 dx}$$
can be obtained from the compact embedding: $H_{\alpha,\mbox{{\tiny rad}}} (B) \hookrightarrow L_l^2 (B)$. Other properties can be obtained by similar arguments  to those in the proof of Theorem \ref{t2.3} and Lemma \ref{l3.1}. \qed

\section{Liouville type results: Proof of Theorems \ref{t0.1} and \ref{t1.1}}
\setcounter{equation}{0}

In this section, we present the proof of Theorems \ref{t0.1} and \ref{t1.1}. We first obtain the
following result.

\begin{lem}
\label{la-4.1}
Assume that $N \geq 5$, \eqref{1.1} holds and $1<p<\min \{\frac{N+4}{N-4}, p_s\}$. Then, there exists a constant $C=C(N,p,\alpha, \tau)>0$ such that the following holds

(i) Any nonnegative solution $u \in C^4 (\R^N \backslash \{0\}) \cap C^0 (\R^N)$ with
 $|x|^\alpha \Delta u \in C^0 (\R^N)$ of (P) in $\Omega=\{x \in \R^N: \; 0<|x|<\rho\} \; (\rho>0)$ satisfies
\begin{equation}
\label{a-4.1}
u(x) \leq C |x|^{-\frac{4+\tau}{p-1}} \;\; \mbox{and} \;\;  |\nabla u(x)| \leq C |x|^{-\frac{p+\tau+3}{p-1}}, \;\;\; 0<|x|<\frac{\rho}{2},
\end{equation}
\begin{equation}
\label{a-4.2}
|v(x)| \leq C |x|^{\alpha-\frac{2(p+1)+\tau}{p-1}} \;\; \mbox{and} \;\;
|\nabla v(x)| \leq C |x|^{\alpha-1-\frac{2(p+1)+\tau}{p-1}}, \;\;\; 0<|x|<\frac{\rho}{2},
\end{equation}
where $v(x):=-|x|^\alpha \Delta u(x)$.

(ii) Any nonnegative solution $u \in C^4 (\R^N \backslash \{0\}) \cap C^0 (\R^N)$ with
$|x|^\alpha \Delta u \in C^0 (\R^N)$ of (P) in $\Omega=\{x \in \R^N: \; |x|>\rho\} \; (\rho > 0)$ satisfies
\begin{equation}
\label{a-4.3}
u(x) \leq C |x|^{-\frac{4+\tau}{p-1}} \;\; \mbox{and} \;\;  |\nabla u(x)| \leq C |x|^{-\frac{p+\tau+3}{p-1}}, \;\;\; |x|>2 \rho,
\end{equation}
\begin{equation}
\label{a-4.4}
|v(x)| \leq C |x|^{\alpha-\frac{2(p+1)+\tau}{p-1}} \;\; \mbox{and} \;\;
|\nabla v(x)| \leq C |x|^{\alpha-1-\frac{2(p+1)+\tau}{p-1}}, \;\;\; |x|>2 \rho.
\end{equation}
\end{lem}

{\bf Proof.} The proof of this lemma is similar to the proof of the first step of Lemma \ref{l3.2}.
For any $x_0 \in \Omega=\{x \in \R^N: \; 0<|x|<\rho\}$ and $0<|x_0|<\rho/2$, or $x_0 \in \Omega=\{x \in \R^N: \; |x|>\rho\}$ and $|x_0|>2 \rho$. We denote $R=\frac{1}{2} |x_0|$ and observe that, for all $y \in B_1$, $\frac{|x_0|}{2}<|x_0+Ry|<\frac{3 |x_0|}{2}$, so that
$x_0+R y \in \Omega$. Let us define
$$U(y)=R^{\frac{4+\tau}{p-1}} u(x_0+Ry), \;\;\; V(y)=R^{\frac{2(p+1)+\tau}{p-1}-\alpha} v(x_0+Ry).$$
Then $(U,V)$ satisfies the system
\begin{equation}
\label{a-4.5}
\left \{ \begin{array}{ll} -\Delta U=|y+\frac{x_0}{R}|^{-\alpha} V \;\; &\mbox{in $B_1$},\\
-\Delta V=|y+\frac{x_0}{R}|^l U^p \;\; &\mbox{in $B_1$}.
\end{array} \right.
\end{equation}
Note that $|y+\frac{x_0}{R}| \in [1,3]$ for all $y \in {\overline {B_1}}$.
Using the fact $1<p<\min \{\frac{N+4}{N-4}, p_s\}$ and the blow-up argument as in the proof of the first step of Lemma \ref{l3.2}, we obtain that
there exists $C>0$ depending only on $N, p, \alpha, l$, independent of $x_0$, such that
\begin{equation}
\label{a-4.6}
U^{\frac{p-1}{4}} (y)+|\nabla U (y)|^{\frac{p-1}{p+3}}+|V (y)|^{\frac{p-1}{2 (p+1)}}
+|\nabla V(y)|^{\frac{p-1}{3p+1}} \leq C (1+\mbox{dist}^{-1} (y, \partial B_1)), \;\; y \in B_1.
\end{equation}
This implies that $U(0)+|\nabla U(0)|+|V(0)|+|\nabla V(0)| \leq C$. This implies that our conclusions in \eqref{a-4.1}, \eqref{a-4.2}, \eqref{a-4.3} and \eqref{a-4.4} hold. \qed

{\bf Proof of Theorem \ref{t0.1}}

We obtain from Lemma \ref{t5.1} in Appendix that, if $u \in C^4(\R^N \backslash \{0\})
\cap C^0(\R^N)$, $v(x):=-|x|^\alpha \Delta u \in C^0 (\R^N)$,
is a nonnegative solution to (P), the following Rellich-Pohozaev identity holds:
\begin{eqnarray}
\label{a-4.6-a}
&& \Bigl[\frac{N'+\tau}{p+1}-\frac{N'-4}{2} \Bigr] \int_{B_R} |x|^l u^{p+1} dx=\int_{\partial B_R} \bigg[R^{l+1} \frac{u^{p+1}}{p+1} \nonumber \\
&&+R^{1-\alpha} \frac{v^2}{2}
+2R u'v' \nonumber -R \nabla u \cdot \nabla v+\frac{N-\alpha}{2} v u'+\frac{N'-4}{2} uv'\bigg]  d \sigma_R
\end{eqnarray}
for all $R>0$. Define
\begin{equation}
\label{a-4.7}
F(R)=\int_{B_R} |x|^l u^{p+1} dx.
\end{equation}
By Rellich-Pohozaev identity, we have
\begin{equation}
\label{a-4.8}
F(R) \leq C \left(G_1 (R)+G_2 (R)\right),
\end{equation}
where
\begin{equation}
\label{a-4.9}
G_1 (R)=R^{N'+\tau} \int_{S^{N-1}} u^{p+1} (R, \theta) d \theta
\end{equation}
and
\begin{eqnarray}
\label{a-4.10}
G_2 (R)=&&R^{N'} \int_{S^{N-1}} \bigg[(\Delta_x u(R, \theta))^2 +R^{-\alpha} \left|\nabla u (R, \theta)\right| \left|\nabla v (R, \theta)\right| \nonumber\\
& &+R^{-(1+\alpha)} \Big(|v (R, \theta)| |\nabla u (R, \theta)|+u (R, \theta) |\nabla v (R, \theta)| \Big) \bigg] d \theta.
\end{eqnarray}
Now, by \eqref{a-4.1}-\eqref{a-4.4} in Lemma \ref{la-4.1}, we have
\begin{equation}
\label{a-4.11}
u(x) \leq C |x|^{-\frac{4+\tau}{p-1}} \;\; \mbox{and} \;\;  |\Delta u(x)| \leq C |x|^{-\frac{2p+2+\tau}{p-1}}, \;\;\; x \neq 0,
\end{equation}
\begin{equation}
\label{a-4.110-a}
|\nabla u(x)| \leq C |x|^{-\frac{p+\tau+3}{p-1}}, \;\;\;|\nabla v(x)| \leq C |x|^{\alpha-1-\frac{2(p+1)+\tau}{p-1}}, \;\;\; x \neq 0,
\end{equation}
\begin{equation}
\label{a-4.110-b}
|v(x)| \leq C  |x|^{\alpha-\frac{2(p+1)+\tau}{p-1}}, \;\;\; x \neq 0.
\end{equation}
Due to $p<p_s$, it follows that
$$G_1 (R)+G_2 (R) \leq C R^{\frac{(N'-4)p-(N'+4+2 \tau)}{p-1}} \longrightarrow 0, \qquad \quad \mbox{as} \  R \to \infty.$$
Therefore, $u \equiv 0$ by \eqref{a-4.8}. This completes the proof of Theorem \ref{t0.1}. \qed

{\bf Proof of Theorem \ref{t0.2}}

A little variant of the proof of Lemma \ref{l4.2} below implies that
$$v(r) \geq 0 \;\;\; \mbox{for $r>0$}$$
provided that $u(r) \geq 0$ for $r>0$. The strong maximum principle implies that $u(r)>0$ and $v(r)>0$
for $r \in [0, \infty)$ provided that $u(r)$ is nontrivial.

We first show that
\begin{equation}
\label{a-4.12}
ru'(r)+(N-2) u (r) \geq 0, \;\;\; r v'(r)+(N-2) v(r) \geq 0, \;\;\; \forall r>0.
\end{equation}

Since $\Delta u (r) \leq 0$ and $\Delta v(r) \leq 0$, we see that
\begin{equation}
\label{a-4.13}
(r u'(r)+(N-2) u(r))' \leq 0, \;\;\; (r v'(r)+(N-2) v(r))' \leq 0, \;\;\; \forall r>0.
\end{equation}
Suppose that there exists $r_0>0$ such that $M_0:=r_0 u'(r_0)+(N-2) u(r_0)<0$, we obtain from
$\eqref{a-4.13}_1$ that
\begin{equation}
\label{a-4.14}
u'(r) \leq r^{-1} [r u'(r)+(N-2) u(r)] \leq M_0 r^{-1}<0 \;\;\; \forall r \geq r_0.
\end{equation}
Integrating \eqref{a-4.14} in $(r_0,r)$ and sending $r$ to $\infty$, we derive a contradiction.
The proof of $\eqref{a-4.12}_2$ is similar.

We now claim that
\begin{equation}
\label{a-4.15}
u(r) \leq C r^{-\frac{4+\tau}{p-1}}, \;\;\;v(r) \leq C r^{\alpha-\frac{2(p+1)+\tau}{p-1}},
\;\;\; \forall r>0,
\end{equation}
\begin{equation}
\label{a-4.16}
|u (r) v' (r)|+|u'(r) v(r)|+|u'(r) v'(r)| \leq C r^{\alpha-1-\frac{2(p+3+\tau)}{p-1}}, \;\;\;
\forall r>0,
\end{equation}
where $C$ is a positive constant depending only on $N,p,\alpha,l$.

Since $(u(r), v(r))$ satisfies the equations:
\begin{equation}
\label{a-4.17}
-(r^{N-1} u'(r))'=r^{N-\alpha-1} v(r), \;\;\; -(r^{N-1} v'(r))'=r^{N+l-1} u^p (r),
\end{equation}
with $r u'(r) \to 0$, $r v'(r) \to 0$ as $r \to 0^+$ (we know that $u$ and $v$ are continuous at $r=0$), we obtain that
\begin{equation}
\label{a-4.18}
-r^{N-1} u'(r) \geq \frac{r^{N-\alpha}}{N-\alpha} v(r), \;\;\;
-r^{N-1} v'(r) \geq \frac{r^{N+l}}{N+l} u^p (r), \;\;\; \forall r>0,
\end{equation}
by integrating both the equations in \eqref{a-4.17} in $(0, r)$ and using the facts $u'(r)<0$ and $v'(r)<0$. Using \eqref{a-4.12}, we have
\begin{equation}
\label{a-4.19}
0 \geq -r u'(r)-(N-2) u(r) \geq \frac{r^{2-\alpha}}{N-\alpha} v(r)-(N-2) u(r),
\end{equation}
\begin{equation}
\label{a-4.20}
0 \geq -r v'(r)-(N-2) v(r) \geq \frac{r^{2+l}}{N+l} u^p (r)-(N-2) v(r).
\end{equation}
These two inequalities imply that the two inequalities in \eqref{a-4.15} hold. Using the two
inequalities in \eqref{a-4.12} and  \eqref{a-4.15}, we obtain \eqref{a-4.16}. Now the proof is finished  by using  Rellich-Pohozaev identity given in Lemma \ref{t5.1}  and the same  arguments as those in the proof of
Theorem \ref{t0.1}. \qed

To prove Theorem \ref{t1.1}, we first obtain the following proposition.

\begin{prop}
\label{p4.1}
Assume that $N \geq 5$, $\alpha=2$, $l>-2$ and $1<p<\frac{N+2+2l}{N-2}$. Let $u \in C^4 (\R^N \backslash \{0\}) \cap C^0 (\R^N)$ and $|x|^2 \Delta u \in C^0(\R^N)$ be a positive solution of (P) in $\R^N$. Then $|x|^{\frac{N-2}{2}} u(x)$ and $|x|^{\frac{N-2}{2}} v(x)$ are strictly increasing in the radius $|x|$, where $v(x)=-|x|^2 \Delta u(x)$.
\end{prop}

To prove this proposition, we first show the following lemma.

\begin{lem}
\label{l4.2}
Assume that $N \geq 5$ and \eqref{1.1} holds with $1<p<p_s$. Assume also that
$u \in C^4 (\R^N \backslash \{0\}) \cap C^0 (\R^N)$ satisfying $|x|^\alpha \Delta u \in C^0(\R^N)$ is a
solution of (P) in $\R^N \backslash \{0\}$
with $u(x)>0$ for $x \in \R^N$. Then $v(x):=-|x|^\alpha \Delta u(x)>0$
for $x \in \R^N$.
\end{lem}

\begin{proof}
The proof of this lemma is quite elementary. We see that $(u,v)$ satisfies the system
\begin{equation}
\label{4.1}
\left \{ \begin{array}{ll} -\Delta u=|x|^{-\alpha} v \;\; &\mbox{in $\R^N$},\\
-\Delta v=|x|^l u^p \;\; &\mbox{in $\R^N$}.
\end{array} \right.
\end{equation}

We consider three cases for $\alpha$: (a) $\alpha=2$, (b) $4-N<\alpha<2$, (c) $2<\alpha<N$.

We give the proof of case (a). The proofs of case (b) and (c) are similar to that of case (a).

{\it Step 1.} We show that $v(0)>0$.

The main idea of the proof is similar to that of Theorem 3.1 of \cite{WX}.
Suppose not, $v(0) \leq 0$. We introduce the average of a function
$${\overline f}(r)=\frac{1}{|\partial B_r (0)|} \int_{\partial B_r (0)} f d \sigma.$$
Then we have by Jensen's inequality
\begin{equation}
\label{4.2}
\left \{ \begin{array}{l} \Delta {\overline u}+r^{-2} {\overline v}=0,\\
\Delta {\overline v}+r^l {\overline u}^p \leq 0,
\end{array} \right.
\end{equation}
where $r=|x|>0$. Since ${\overline v}(0)=v(0) \leq 0$ and ${\overline v}'(r)<0$ for $r>0$, we
have
\begin{equation}
\label{4.3}
{\overline v}(r)<{\overline v}(0) \leq 0 \;\; \mbox{for all $r>0$}.
\end{equation}
Hence
$$\Delta {\overline u}>0 \;\;\; \mbox{for $r>0$}.$$
This implies that
\begin{equation}
\label{4.4}
{\overline u}'(r)>0 \;\; \mbox{for $r>0$}.
\end{equation}
On the other hand, we see from \eqref{4.3} that for a fixed $\kappa_0>0$,
$${\overline v}(r)<{\overline v}(\kappa_0)<0 \;\; \mbox{for $r>\kappa_0$}.$$
Then
$$ \Delta {\overline u} \geq r^{-2} (-{\overline v}(\kappa_0)) \;\; \mbox{for $r \geq \kappa_0$}$$
and
$$ r^{N-1} {\overline u}'(r)-\kappa_0^{N-1} {\overline u}'(\kappa_0)
\geq \frac{(-{\overline v} (\kappa_0))}{N-2} \Big[ r^{N-2}-\kappa_0^{N-2} \Big].$$
It follows from \eqref{4.4} that
$$r^{N-1} {\overline u}'(r)
\geq \frac{(-{\overline v} (\kappa_0))}{N-2} \Big[ r^{N-2}-\kappa_0^{N-2} \Big]$$
and
$${\overline u}(r) \geq \frac{(-{\overline v} (\kappa_0))}{N-2} \ln \frac{r}{\kappa_0}-\frac{(-{\overline v} (\kappa_0))}{(N-2)^2} \geq C_0>1
\;\; \mbox{for $r \geq r_0>\kappa_0>0$},$$
where $r_0>\kappa_0$ is a suitably large constant.

Suppose now that
$${\overline u}(r) \geq \frac{(C_0)^{p^k}}{A^{b_k}} r^{\sigma_k} \;\; \mbox{for $r \geq r_k>r_0$},$$
where
$$A=[N+l+2p M]^{\frac{8}{2+l}}, \;\; \sigma_0=b_0=0$$
and $M>1$, $\sigma_k \geq 0$, $b_k \geq 0$ for $k \geq 1$ are determined later. We consider two cases: (i) $l+p \sigma_k \neq -1$ for all $k=0,1,
\ldots$, (ii) there is an integer $k_0 \geq 0$ such that $l+p \sigma_{k_0}=-1$.

For the case (i), we have
\begin{equation*}
\begin{split}
r^{N-1} {\overline v}'(r) &\leq r_k^{N-1} {\overline v}'(r_k)-\int_{r_k}^r s^{N+l-1} {\overline u}^p (s) ds\\
 &\leq -\frac{(C_0)^{p^{k+1}}}{A^{b_k p} (N+l+p \sigma_k)}
\Big[ r^{N+l+p \sigma_k}-r_k^{N+l+p \sigma_k} \Big]\\
&\le -\frac{(C_0)^{p^{k+1}}}{2A^{b_k p} (N+l+p \sigma_k)}
 r^{N+l+p \sigma_k}.
 \end{split}
 \end{equation*}
Hence
$${\overline v}'(r) \leq -\frac{(C_0)^{p^{k+1}}}{2 A^{b_k p} (N+l+p \sigma_k)} r^{1+l+p \sigma_k}$$
for $r \geq 2^{\frac{1}{N+l+p \sigma_k}} r_k$. Similarly
$${\overline v}(r) \leq -\frac{(C_0)^{p^{k+1}}}{4 A^{b_k p} (N+l+p \sigma_k)(2+l+p \sigma_k)}
r^{2+l+p \sigma_k}$$
for $r \geq 2^{\frac{1}{2+l+p \sigma_k}} 2^{\frac{1}{N+l+p \sigma_k}} r_k$.
Hence
$${\overline v}(r) \leq -\frac{(C_0)^{p^{k+1}}}{A^{b_k p} 4 (N+l+p \sigma_k)^2} r^{2+l+p \sigma_k}.$$
Note that
$$2+l+p \sigma_k>0$$
since $l>-2$. Therefore,
$${\overline u}(r) \geq \frac{(C_0)^{p^{k+1}}}{A^{b_k p} 16 (N+l+p \sigma_k)^4} r^{2+l+p \sigma_k}$$
for $r \geq 2^{\frac{4}{2+l+p \sigma_k}} r_k>r_k$.

Set
$$ \begin{array}{l} \sigma_0=0,\\
\sigma_{k+1}=2+l+p \sigma_k, \\
r_{k+1}=2^{\frac{4}{2+l+p \sigma_k}}r_k.
\end{array}$$
First of all, by mathematical induction, it is easy to see that
$$2^4 (N+l+p \sigma_k)^4 \leq A^{(2+l)(k+1)}$$
if we choose $M>1$ sufficiently large. Note that
\begin{equation}
\label{4.5}
\sigma_k=\frac{2+l}{p-1} \Big(p^{k}-1 \Big).
\end{equation}
Hence we can set
$$b_0=0, \;\;\; b_{k+1}=p b_k+(2+l) (k+1).$$
Then we have
\begin{equation}
\label{4.6}
{\overline u}(r) \geq \frac{(C_0)^{p^{k+1}}}{A^{b_{k+1}}} r^{\sigma_{k+1}} \;\;\mbox{for $r \geq r_{k+1}$}.
\end{equation}
Notice that
\begin{equation}
\label{4.6-1}
r_{k+1} \leq c r_0,
\end{equation}
where $c$ can be chosen to be $2^{\Sigma_{k=0}^\infty \frac{4}{2+l+p \sigma_k} }$. Also notice that, by using the iteration formulas above, we have
$$b_k=\frac{(2+l)}{(p-1)^2} \Big[p^{k+1}-(k+1)p+k \Big].$$
Hence, if we take $M>1$ large enough so that
$$A^{\frac{p}{p-1}}> 2c r_0$$
and then take ${\tilde r}_0=A^{\frac{p}{p-1}}$, we see
$${\overline u}({\tilde r}_0) \geq (C_0)^{p^{k+1}} A^{\frac{(2+l)}{p-1}(k+1)} \to \infty\;\;
\mbox{as $k \to \infty$}.$$
Since ${\tilde r}_0$ is independent of $k$, a contradiction is reached.

For the case (ii), we see that $l+p \sigma_k \neq -1$ for $k \neq k_0$. We only need to deal with
$k=k_0$, since $1+l+p \sigma_{k_0}=0$.
Arguments similar to those in the  case (i) imply that
$$r^{N-1}{\overline v}'(r) \leq -\frac{C_0^{p^{k_0+1}}}{2 (N-1) A^{b_{k_0} p}} r^{N-1}
\;\; \mbox{for $r \geq 2^{\frac{1}{N-1}} r_{k_0}$}.$$
Similarly,
$${\overline v} (r) \leq -\frac{(C_0)^{p^{k_0+1}}}{4 (N-1) A^{b_{k_0} p}} r \;\; \mbox{for $r \geq 2
\times 2^{\frac{1}{N-1}} r_{k_0}$},$$
$${\overline u}(r) \geq \frac{(C_0)^{p^{k_0+1}}}{2^4 (N-1)^2 A^{b_{k_0} p}} r \;\; \mbox{for $r \geq 2^2 \times 2^{\frac{2}{N-1}} r_{k_0}$}.$$
This implies that
$${\overline u}(r) \geq \frac{(C_0)^{p^{k_0+1}}}{2^4 (N-1)^4 A^{b_{k_0} p}} r \;\; \mbox{for $r \geq 2^4 \times r_{k_0}$}.$$
Now, we only need to make the change $r_{k_0+1}=2^4 \times r_{k_0}$ and choose the constant $c$ in \eqref{4.6-1} to be $2^{\Sigma_{k \neq k_0}  \frac{4}{2+l+p \sigma_k} +4}$, we can derive a contradiction by arguments exactly the same as those in the  case (i). Note that $N+l+p \sigma_{k_0}=N-1$, $2+l+p \sigma_{k_0}=1$. This completes the proof of Step 1.

{\it Step 2.} We claim
\begin{equation}
\label{4.7}
v(x)>0 \;\;\mbox{for $x \in \R^N$}.
\end{equation}

Suppose not, there is $x_0 \neq 0$ such that $v(x_0) \leq 0$. Set
$$w(y)=u(x), \;\; z(y)=v(x), \;\; y=x-x_0.$$
We have that $(w,z)$ satisfies the system
\begin{equation}
\label{4.8}
\left \{ \begin{array}{ll} \Delta w+|y+x_0|^{-2} z=0 \;\; &\mbox{in $\R^N$},\\
\Delta z+|y+x_0|^l w^p=0 \;\; &\mbox{in $\R^N$},
\end{array} \right.
\end{equation}
and
\begin{equation}
\label{4.9}
\left \{ \begin{array}{l} \Delta {\overline w}+{\overline {|y+x_0|^{-2} z}}=0,\\
\Delta {\overline z}+{\overline {|y+x_0|^l w^p}}=0.
\end{array} \right.
\end{equation}
It follows from the second equation of \eqref{4.9} that
\begin{equation}
\label{4.10}
{\overline z}'(\rho)<0, \;\; {\overline z} (\rho)<{\overline z}(0) \leq 0 \;\; \mbox{for $\rho>0$},
\end{equation}
where $\rho=|y|$. These imply that there is $-\infty \leq \vartheta<0$ such that
$${\overline z}(\rho) \to \vartheta \;\; \mbox{as $\rho \to \infty$}.$$
For $\rho>10 |x_0|$ and sufficiently large, it is known from the first equation of \eqref{4.9} that
\begin{equation}
\label{4.11}
-\Delta {\overline w}=\rho^{-2} {\overline {\Big|\frac{y}{\rho}+\frac{x_0}{\rho} \Big|^{-2} z}}.
\end{equation}
Since $\Big|\frac{y}{\rho}+\frac{x_0}{\rho} \Big| \to 1$ as $\rho \to \infty$, we see that
$${\overline {\Big|\frac{y}{\rho}+\frac{x_0}{\rho} \Big|^{-2} z}} (\rho) \to \vartheta \;\; \mbox{as $\rho \to \infty$}.$$
This implies that there are $\vartheta<{\tilde \vartheta}<\frac{1}{2}\vartheta$ and ${\overline \kappa}_0>0$ such that
\begin{equation}
\label{4.12}
\Delta {\overline w} \geq (-{\tilde \vartheta}) \rho^{-2} \;\; \mbox{for $\rho \geq {\overline \kappa}_0$}.
\end{equation}
Arguments similar to those in the proof of step 1 imply that we can reach a contradiction.
Note that it is seen from the second equation of \eqref{4.9} that
\begin{equation}
\label{4.13}
-\Delta {\overline z}=\rho^l {\overline {\Big|\frac{y}{\rho}+ \frac{x_0}{\rho} \Big|^l w^p}}
\geq \frac{1}{2} \rho^l {\overline w}^p
\end{equation}
for $\rho \geq {\overline \kappa}_0$, if we choose ${\overline \kappa}_0$ large enough. By arguments
exactly same as those in the proof of step 1 and by choosing $A$ as in the proof of step 1 such that
$$2^5 (N+l+p \sigma_k)^4 \leq A^{(2+l)(k+1)},$$
we can reach a contradiction.
This implies that  $v(x_0) \leq 0$ does not exist and our claim \eqref{4.7} holds. This completes the proof of case (a).

For the case (b), arguments similar to those in the proof of the case (a) implies that
$${\overline u}(r) \geq C r^{2-\alpha} \geq C_0>1
\;\; \mbox{for $r \geq r_0>0$},$$
where $r_0>0$ is a suitably large constant. By arguments similar to those in the proof of the case (a), we obtain that
$${\overline u}(r) \geq \frac{C_0^{p^k}}{A^{b_k}} r^{\sigma_k} \;\; \mbox{for $r \geq r_k>r_0$},$$
where
$$A=[N+2+\tau+2p M]^{\frac{8}{4+\tau}}, \;\; \sigma_0=b_0=0$$
and $M>1$ with
$$ \begin{array}{l} \sigma_0=0,\\
\sigma_{k+1}=4+\tau+p \sigma_k, \\
r_{k+1}= 2^{\frac{4}{3+\tau+p \sigma_k}} r_k.
\end{array}$$
Therefore, we have
$$\sigma_k=\frac{4+\tau}{p-1} \Big(p^{k}-1 \Big).$$
(Note that $4+\tau>0$.)
Hence we can set
$$b_0=0, \;\;\; b_{k+1}=p b_k+(4+\tau) (k+1)$$
and
$$b_k=\frac{(4+\tau)}{(p-1)^2} \Big[p^{k+1}-(k+1)p+k \Big].$$
Moreover, we also need to consider $3+\tau+p\sigma_k=0$ or not. The main idea is similar to
that of the proof of the case (a). This completes the proof of the case (b).

The proof of the case (c) is exactly same as that of the case (b). Note that for $\alpha>2$,
we see that $2+l>0$ since $l>\alpha-4$. Arguments similar to those in the proof of case (a) imply
\begin{equation}
\label{4.13-1}
{\overline u}(r) \geq C (-{\overline v}(\kappa_0)) \kappa_0^{2-\alpha}
\;\; \mbox{for $r \geq r_0>\kappa_0>0$},
\end{equation}
where $r_0>\kappa_0>0$ is a suitably large constant. By a simple calculation, we easily know
that $-{\overline v} (\kappa_0) \geq C \kappa_0^{2+l}$ if we choose a sufficiently large $\kappa_0>0$ (note that $2+l>0$). This and \eqref{4.13-1} imply that
\begin{equation}
\label{4.13-2}
{\overline u}(r) \geq C  \kappa_0^{4+\tau} \geq C_0>1
\;\; \mbox{for $r \geq r_0>\kappa_0>0$}.
\end{equation}
Choosing $A$, $\sigma_k$, $b_k$ exactly the same as that in  the case (b) and
using the same arguments, we complete the proof of the case (c).
\end{proof}

{\bf Proof of Proposition \ref{p4.1}}

We introduce the following transform
$$w(t, \omega)=r^{\frac{N-2}{2}} u(r, \omega), \;\;\; z(t, \omega)=r^{\frac{N-2}{2}} v(r, \omega), \;\; t=\ln r$$
with $r=|x|$. Then $(w, z)$ satisfies the system of equations
\begin{equation}
\label{4.14}
\left \{ \begin{array}{ll} w_{tt}+\Delta_{S^{N-1}} w-\frac{(N-2)^2}{4}  w+z=0, \;\; &(t, \omega)
\in (-\infty, \infty) \times S^{N-1}, \\
z_{tt}+\Delta_{S^{N-1}} z-\frac{(N-2)^2}{4} z+e^{p^* t} w^p=0, \;\; & (t, \omega)
\in (-\infty, \infty) \times S^{N-1},
\end{array} \right.
\end{equation}
where
$$p^*=\frac{1}{2} \Big[(N+2+2l)-(N-2)p \Big].$$
Note that $p^*>0$ for $1<p<\frac{N+2+2 l}{N-2}\; (=\frac{N'+4+2 \tau}{N'-4}:=p_s)$.
Moreover, we have
\begin{equation}
\label{4.15}
\lim_{t \to -\infty} e^{-\frac{N-2}{2} t} w(t, \omega)=u(0), \;\;\lim_{t \to -\infty} e^{-\frac{N-2}{2} t} z(t, \omega)=v(0)
\end{equation}
uniformly on $S^{N-1}$. Note that $0<u(0)<\infty$, $0<v(0)<\infty$ since $u$ and $v$ are continuous at $x=0$.

Let
$$\Sigma=\{(t, \omega): \; -\infty<t<\infty, \; \omega \in S^{N-1}\}.$$
For $T \in \R$ we define
$$\Sigma_T= (-\infty,T) \times S^{N-1}, \;\; S_T=\{T\} \times S^{N-1}.$$
Further, we let $(t, \omega)_T$ be the reflection of $(t, \omega)$ with respect to $\{T\} \times S^{N-1}$, namely,
$$(t, \omega)_T=(2T-t, \omega).$$
Then the functions
$$w_T (t, \omega)=w(2T-t, \omega), \qquad z_T (t, \omega)=z(2T-t, \omega),$$
and
$${\tilde w}_T=w-w_T, \qquad {\tilde z}_T=z-z_T$$
are well defined in $\Sigma_T$ for $T>-\infty$. Moreover, $({\tilde w}_T, {\tilde z}_T)$ satisfies the following system
in $\Sigma_T$
\begin{equation}
\label{4.15-1}
\left \{ \begin{array}{l} ({\tilde w}_T)_{tt}+\Delta_{S^{N-1}} {\tilde w}_T-\frac{(N-2)^2}{4}  {\tilde w}_T+{\tilde z}_T=0, \\
({\tilde z}_T)_{tt}+\Delta_{S^{N-1}} {\tilde z}_T-\frac{(N-2)^2}{4} {\tilde z}_T+e^{p^* t} w^p
-e^{p^* (2T-t)} (w_T)^p
=0.
\end{array} \right.
\end{equation}
For fixed $\omega \in S^{N-1}$, we have that if $1<p<p_s$,
$$e^{p_* (2T-t)} (w_T)^p \geq e^{p_*t} (w_T)^p \;\; \mbox{for $t \leq T$}.$$
In turn, we have
\begin{equation}
\label{4.16}
\left \{ \begin{array}{l} ({\tilde w}_T)_{tt}+\Delta_{S^{N-1}} {\tilde w}_T-\frac{(N-2)^2}{4}  {\tilde w}_T+{\tilde z}_T=0, \\
({\tilde z}_T)_{tt}+\Delta_{S^{N-1}} {\tilde z}_T-\frac{(N-2)^2}{4} {\tilde z}_T+e^{p^* t} [w^p-(w_T)^p]_+
 \geq 0
\end{array} \right.
\end{equation}
in $\Sigma_T$, where we write $[h (t, \omega)]_+=\max \{0, h (t, \omega)\}$. Therefore, $({\tilde w}_T, {\tilde z}_T)$ satisfies
\begin{equation}
\label{4.17}
\left \{ \begin{array}{ll} ({\tilde w}_T)_{tt}+\Delta_{S^{N-1}} {\tilde w}_T-\frac{(N-2)^2}{4}  {\tilde w}_T+{\tilde z}_T=0, \;\; &\mbox{in $\Sigma_T$},\\
({\tilde z}_T)_{tt}+\Delta_{S^{N-1}} {\tilde z}_T-\frac{(N-2)^2}{4} {\tilde z}_T+c(t, \omega) {\tilde w}_T
 \geq 0, \;\; &\mbox{in $\Sigma_T$},\\
 {\tilde w}_T \leq 0, \;\; {\tilde z}_T \leq 0,  \;\; &\mbox{on $\partial \Sigma_T$},
\end{array} \right.
\end{equation}
where
$$c(t, \omega)=e^{p_* t} \frac{[w^p-(w_T)^p]_+}{w-(w_T)}.$$
The fact that $({\tilde w}_T, {\tilde z}_T) \leq 0$ on $\partial \Sigma_T$ can be obtained from $(w(t, \omega), z(t, \omega))
\to (0,0)$ as $t \to -\infty$ uniformly for $\omega \in S^{N-1}$ and $(w (t, \omega), z(t, \omega))>0$ for
$(t, \omega) \in (-\infty, \infty) \times S^{N-1}$. Note first that by \eqref{4.15} the functions
$e^{-\frac{N-2}{2} t} w$ and $e^{-\frac{N-2}{2} t} z$ are bounded in $\Sigma_T$ uniformly in $T<{\overline T}$, where ${\overline T}$ is a sufficiently negative number. We see that there exists $L=L({\overline T})>0$
such that, if we define $\Sigma_T^+=\{(t, \omega) \in \Sigma_T: {\tilde w}_T (t, \omega)>0\}$,
\begin{equation}
\label{4.18}
0  \leq c(t, \omega) \leq pL e^{(2+l) t} \;\; \mbox{on $\Sigma_T^+$ and $T<{\overline T}$}.
\end{equation}
Therefore, we can choose a sufficiently negative number ${\hat T} \ll -1$ such that
$$c(t, \omega)<\frac{(N-2)^2}{4}$$
uniformly for $(t, \omega) \in \Sigma_T^+$ and $T \leq {\hat T}$. This implies that
\begin{equation}
\label{4.19}
c (t, \omega)-\frac{(N-2)^2}{4}<0 \;\; \mbox{for $ (t, \omega) \in \Sigma_T$ and $T \leq {\hat T}$}.
\end{equation}
By \eqref{4.15}, one has
$$\lim_{t \to -\infty} \sup {\tilde w}_T (t, \omega) \leq 0,\;\;\;\lim_{t \to -\infty} \sup {\tilde z}_T (t, \omega) \leq 0,$$
and
$${\tilde w}_T \equiv 0, \;\;\; {\tilde z}_T \equiv 0 \;\; \mbox{on $S_T$}.$$
By the maximum principle of the cooperative systems (see Lemma 11 of \cite{RZ}), in view of \eqref{4.17} and \eqref{4.19}, it follows that $({\tilde w}_T, {\tilde z}_T)$ must be non-positive, provided $T \leq {\hat T}$. Note
that $1-\frac{(N-2)^2}{4}<0$. Clearly
\begin{equation}
\label{4.20}
\frac{\partial {\tilde w}_T}{\partial t} (T, \omega)=2 \frac{\partial w}{\partial t} (T, \omega) \geq 0, \;\; \frac{\partial {\tilde z}_T}{\partial t} (T, \omega)=2 \frac{\partial z}{\partial t} (T, \omega) \geq 0, \;\;\;  T
\leq {\hat T}.
\end{equation}
Moreover, for any $T<{\hat T}$, either
\begin{equation}
\label{4.21}
{\tilde w}_T (t, \omega)<0 \;\; \mbox{in $\Sigma_T$}, \;\;\; \frac{\partial {\tilde w}_T}{\partial t} (t, \omega)>0
\;\; \mbox{on $S_T$},
\end{equation}
and
\begin{equation}
\label{4.22}
{\tilde z}_T (t, \omega)<0 \;\; \mbox{in $\Sigma_T$}, \;\;\; \frac{\partial {\tilde z}_T}{\partial t} (t, \omega)>0
\;\; \mbox{on $S_T$},
\end{equation}
by the strong maximum principle and the boundary lemma or one of two cases occurs: ${\tilde w}_T \equiv 0, {\tilde z}_T \equiv 0$
in $\Sigma_{\hat T}$.
One readily sees that \eqref{4.21} and \eqref{4.22} hold for $T<0$ negative enough. For otherwise, one deduces that $w \equiv c_1$ or $z \equiv c_2$ for some constants $c_1$ and $c_2$ and $T<0$ sufficiently negative, thanks to the monotonicity \eqref{4.20}. This is clearly impossible. Hence \eqref{4.21} and \eqref{4.22} hold for sufficiently negative $T$.

Define
\begin{equation}
\label{4.23}
T_0=\sup \{T_1>-\infty: \; \mbox{both \eqref{4.21} and \eqref{4.22} hold for $T \leq T_1$} \}.
\end{equation}
Arguments as the above imply that $T_0$ is well defined.

Obviously we only need to show $T_0=\infty$. Suppose for contradiction that $T_0<\infty$. We want to show that
\begin{equation}
\label{4.24}
{\tilde w}_{T_0} (t, \omega) \equiv 0 \;\; \mbox{or} \;\; {\tilde z}_{T_0} (t, \omega) \equiv 0,  \;\; \mbox{for $(t, \omega) \in \Sigma_{T_0}$}
\end{equation}
and for all $T \in (-\infty, T_0)$,
\begin{equation}
\label{4.25}
{\tilde w}_T (t, \omega)<0 \;\; \mbox{for $(t, \omega) \in \Sigma_T$},
\;\; \frac{\partial {\tilde w}_T}{\partial t} (t, \omega)>0 \;\; \mbox{for $(t, \omega) \in S_T$},
\end{equation}
and
\begin{equation}
\label{4.26}
{\tilde z}_T (t, \omega)<0 \;\; \mbox{for $(t, \omega) \in \Sigma_T$},
\;\; \frac{\partial {\tilde z}_T}{\partial t} (t, \omega)>0 \;\; \mbox{for $(t, \omega) \in S_T$}.
\end{equation}
It is easily seen that \eqref{4.24}-\eqref{4.26} are impossible since the equation of $(w,z)$
contains the term $e^{p_* t}$, which is strictly increasing with respect to $t$. This implies that
$T_0=\infty$, if we have shown \eqref{4.24}-\eqref{4.26}.

We also use contradiction arguments to show \eqref{4.24}-\eqref{4.26}.
Suppose that \eqref{4.24} does not hold. Then by the strong maximum principle and the boundary
lemma,
\begin{equation}
\label{4.27}
{\tilde w}_{T_0}<0 \;\; \mbox{in $\Sigma_{T_0}$}, \;\; \frac{\partial {\tilde w}_{T_0}}{\partial t}>0
\;\; \mbox{on $S_{T_0}$},
\end{equation}
and
\begin{equation}
\label{4.28}
{\tilde z}_{T_0}<0 \;\; \mbox{in $\Sigma_{T_0}$}, \;\; \frac{\partial {\tilde z}_{T_0}}{\partial t}>0
\;\; \mbox{on $S_{T_0}$}.
\end{equation}
These and the definition of $T_0$ imply that
$$\frac{\partial w}{\partial t} (t, \omega)>0, \;\; \frac{\partial z}{\partial t} (t, \omega)>0, \;\; \mbox{for $(t, \omega) \in (-\infty, T_0] \times
S^{N-1}$}.$$
In particular, by the compactness of $S^{N-1}$ and continuity, there exists $\epsilon_0>0$ such that
$$\frac{\partial w}{\partial t} (t, \omega)>0,\;\; \frac{\partial z}{\partial t} (t, \omega)>0, \;\; \mbox{for $(t, \omega) \in (-\infty, T_0+\epsilon_0] \times
S^{N-1}$}.$$
Next we choose from \eqref{4.18} a sufficiently negative value $T_1<T_0$ such that for all $T \leq
T_0+\epsilon_0$, we have
$$c(t, \omega)-\frac{(N-2)^2}{4}<0 \;\; \mbox{in $\Sigma_T^+ \cap \{t<T_1\}$}.$$
Furthermore, thanks to the fact that $S^{N-1}$ has no boundary and by continuity again, there exists $0<\epsilon_1<\epsilon_0$ such that
$${\tilde w}_T (t, \omega)<0, \;\; {\tilde z}_T (t, \omega)<0, \;\; \mbox{in $\Sigma_T \cap \{T_1 \leq t<T\}$ for all $T \leq T_0+\epsilon_1$}.$$
In turn, it follows that
$$c(t, \omega) \equiv 0 \;\; \mbox{in $\Sigma_T \cap \{T_1 \leq t<T\}$ for all $T \leq T_0+\epsilon_1$}.$$
In particular, for all $T \leq T_0+\epsilon_1$, \eqref{4.17} holds with
$$c(t, \omega)-\frac{(N-2)^2}{4}<0 \;\; \mbox{in $\Sigma_T$}.$$
Thus the maximum principle implies again that $({\tilde w}_T, {\tilde z}_T)$ can not have a positive maximum in $\Sigma_T$,
provided $T \leq T_0+\epsilon_1$. Hence both \eqref{4.25} and \eqref{4.26} hold for $T<T_0+\epsilon_1$ by the strong
maximum principle and the boundary lemma. This is again a contradiction to the definition of $T_0$
and therefore \eqref{4.24} must hold. Clearly \eqref{4.25} and \eqref{4.26} are a direct consequence of \eqref{4.24} and this finishes the proof. \qed

{\bf Proof of Theorem \ref{t1.1}}

Suppose that (P) admits a nontrivial nonnegative solution $u \in C^4 (\R^N \backslash \{0\})
\cap C^0 (\R^N)$ and $v(x):=-|x|^2 \Delta u \in C^0 (\R^N)$. Arguments similar to those in the proof of Lemma \ref{l4.2} imply that $v \geq 0$ in $\R^N$. The continuity of $u$ and $v$ at $x=0$ and the strong maximum principle implies
that $u>0$ and $v>0$ on $\R^N$.

Let $R>1$ and $\mathcal{F}_R=\{x \in \R^N: \; |x| \geq R\}$.
We claim that for any $R_0>1$ and $x \in {\mathcal F}_{R_0}$,
\begin{equation}
\label{4.29}
u (x) \geq u_0 \Big(\frac{R_0}{|x|} \Big)^{\frac{N-2}{2}}, \;\;\; v (x) \geq v_0 \Big(\frac{R_0}{|x|} \Big)^{\frac{N-2}{2}},
\end{equation}
where
$$u_0=\min_{|x|=R_0} u(x)>0, \;\;\; v_0=\min_{|x|=R_0} v(x)>0.$$

Obviously, by Proposition \ref{p4.1}, we have for $t>-\infty$,
$$0<\frac{\partial w}{\partial t} (t, \omega)=e^{\frac{N}{2} t} \Big(\frac{N-2}{2} \frac{u}{r}
+\frac{\partial u}{\partial r} (r, \omega) \Big), \;\; (t, \omega) \in S_t$$
and
$$0<\frac{\partial z}{\partial t} (t, \omega)=e^{\frac{N}{2} t} \Big(\frac{N-2}{2} \frac{v}{r}
+\frac{\partial v}{\partial r} (r, \omega) \Big), \;\; (t, \omega) \in S_t$$
where $r=|x|$. In turn
\begin{equation}
\label{4.30}
\frac{N-2}{2r}>-u^{-1} \frac{\partial u}{\partial r} (r, \omega), \;\; \frac{N-2}{2r}>-v^{-1} \frac{\partial v}{\partial r} (r, \omega),\;\; x \in \R^N \backslash \{0\}.
\end{equation}
Integrating \eqref{4.30} from $R_0$ to $r>R_0$ along the radius immediately yields
our claim \eqref{4.29}.

For any $R>R_0$, set $\Omega_R=\{x \in \R^N: \; ||x|-3R|<R\}$. We see that
$\Omega_R \subset \mathcal{F}_{R_0}$ and
\begin{equation}
\label{4.31}
|x|^l u^p \geq D R^l u^p, \;\; |x|^{-2} v \geq D R^{-2} v, \;\;\; \mbox{in $\Omega_R$},
\end{equation}
where $D=  4^{-2}>0$.

For any $0<\delta \leq 1$, consider the problem
\begin{equation}
\label{4.32}
\left \{ \begin{array}{ll} h_{ss}+\frac{N-1}{s+\frac{3}{\delta}} h_s+D k=0 \;\;& \mbox{in $(-1,1)$},\\
k_{ss}+\frac{N-1}{s+\frac{3}{\delta}} k_s+D h^p=0 \;\;& \mbox{in $(-1,1)$},\\
h(-1)=h(1)=k(-1)=k(1)=0. \end{array} \right.
\end{equation}
Direct calculation yields that, for suitably large $a>0$ and $K>0$ independent of $\delta$,
$({\tilde h} (s), {\tilde k} (s)) \equiv (K (1-s^2)^a, K^\varrho (1-s^2)^a)$ with $1<\varrho<p$ is a subsolution to \eqref{4.32}, i.e., for any
$0<\delta \leq 1$,
\begin{equation}
\label{4.33}
\left \{ \begin{array}{ll} {\tilde h}_{ss}+\frac{N-1}{s+\frac{3}{\delta}} {\tilde h}_s+D {\tilde k} \geq 0 \;\;& \mbox{in $(-1,1)$},\\
{\tilde k}_{ss}+\frac{N-1}{s+\frac{3}{\delta}} {\tilde k}_s+D {\tilde h}^p \geq 0 \;\;& \mbox{in $(-1,1)$},\\
{\tilde h}(-1)={\tilde h}(1)={\tilde k}(-1)={\tilde k}(1)=0. \end{array} \right.
\end{equation}

Define
$$u_R (x)=u_R (|x|):=R^{-\beta} {\tilde h} \Big(\frac{|x|-3R}{R} \Big), \;\;\;\; x \in \Omega_R=\{x: \; 2R<|x|<4R\},$$
$$v_R (x)=v_R (|x|):=R^{-\beta} {\tilde k} \Big(\frac{|x|-3R}{R} \Big), \;\;\;\; x \in \Omega_R=\{x: \; 2R<|x|<4R\},$$
where
\begin{equation}
\label{4.34}
\beta:=\frac{2+l}{p-1}>\frac{N-2}{2}.
\end{equation}
In view of \eqref{4.29} and \eqref{4.34}, we immediately deduce that there exists a (sufficiently large) value $R_1 \gg 1$ such that
\begin{equation}
\label{4.36}
u_{R_1}<u, \;\; v_{R_1}<v \;\;\mbox{in $\Omega_{R_1}$}.
\end{equation}
By \eqref{4.31} and the fact that $\beta (p-1)-2=l$, we obtain the following two inequalities
\begin{equation}
\label{4.37}
\left \{ \begin{array}{ll} \Delta u+D R_1^{-2} v \leq 0 \leq \Delta u_{R_1}+D R_1^{-2} v_{R_1} \;\;&\mbox{in $\Omega_{R_1}$},\\
\Delta v+D R_1^l u^p \leq 0 \leq \Delta v_{R_1}+D R_1^l u_{R_1}^p \;\; &\mbox{in $\Omega_{R_1}$}.
\end{array} \right.
\end{equation}
Next, for $0<\delta \leq 1$, we define the functions
$$w_\delta (x)=w_\delta (|x|):=\delta^{-\gamma} R_1^{-\beta} {\tilde h} \Big(\frac{|x|-3R_1}{\delta R_1} \Big),$$
$$z_\delta (x)=z_\delta (|x|):=\delta^{-(\gamma+2)} R_1^{-\beta} {\tilde k} \Big(\frac{|x|-3R_1}{\delta R_1} \Big)$$
in ${\tilde \Omega}_{\delta R_1}=\{x: \; ||x|-3R_1|<R_1 \delta\}$, where
$$\gamma=\frac{4}{p-1}>0.$$
A straight-forward computation implies
\begin{equation}
\label{4.38}
\left \{ \begin{array}{ll} \Delta w_\delta+D R_1^{-2} z_\delta \geq 0 \;\; &
\mbox{in ${\tilde \Omega}_{\delta R_1}$}, \\
\Delta z_\delta+D R_1^l w_\delta^p \geq 0 \;\; &
\mbox{in ${\tilde \Omega}_{\delta R_1}$}, \\
w_\delta=0, \;\; z_\delta=0 \;\; & \mbox{on $\partial {\tilde \Omega}_{\delta R_1}$}.
\end{array} \right.
\end{equation}
Notice that $(w_1, z_1) \equiv (u_{R_1}, v_{R_1})$.

On the other hand, clearly $w_\delta (3R_1) \to \infty$, $z_\delta (3R_1) \to \infty$ as $\delta \to 0$. Therefore, there exists $\delta \in (0,1)$ and a point ${\overline x} \in {\tilde \Omega}_{\delta R_1}$ such that
\begin{equation}
\label{4.39}
u  \geq w_\delta, \;\; v \geq z_\delta \;\;\; \mbox{in ${\tilde \Omega}_{\delta R_1}$ and $u({\overline x})=w_\delta ({\overline x})$}
\end{equation}
or
\begin{equation}
\label{4.39}
u  \geq w_\delta, \;\; v \geq z_\delta\;\;\; \mbox{in ${\tilde \Omega}_{\delta R_1}$ and $v({\overline x})=z_\delta ({\overline x})$}.
\end{equation}
In view of \eqref{4.38}, recall that $(u,v)$ is a supersolution and $(w_\delta,z_\delta)$ is a subsolution to
the problem:
$$
\left \{ \begin{array}{ll} \Delta f+D R_1^{-2} g=0 \;\; &
\mbox{in ${\tilde \Omega}_{\delta R_1}$}, \\
\Delta g+D R_1^l f^p=0 \;\; &
\mbox{in ${\tilde \Omega}_{\delta R_1}$}, \\
f=g=0 \;\; & \mbox{on $\partial {\tilde \Omega}_{\delta R_1}$}.
\end{array} \right.
$$
The strong maximum principle implies $u \equiv w_\delta$ or $v \equiv z_\delta$ in ${\overline {{\tilde \Omega}_{\delta R_1}}}$, which is impossible (since $(u,v)>0$ on ${\overline {{\tilde \Omega}_{\delta R_1}}}$, but $(w_\delta, z_\delta) \equiv 0$ on $\partial {\tilde \Omega}_{\delta R_1}$). This contradiction completes the proof of our Theorem \ref{t1.1}. \qed

\section{Appendix}
\setcounter{equation}{0}

We present  Rellich-Pohozaev identity for a nonnegative solution $u \in C^4 (\R^N \backslash \{0\}) \cap C^0 (\R^N)$ and $|x|^\alpha \Delta u \in C^0 (\R^N)$ of (P).

\begin{lem}
\label{t5.1}
Assume $N \geq 5$, $p>1$ and \eqref{1.1} holds. Assume also that $u \in C^4(\R^N \backslash \{0\})
\cap C^0(\R^N)$ with $v(x):=-|x|^\alpha \Delta u \in C^0 (\R^N)$
is a nonnegative solution to (P), the following Rellich-Pohozaev identity holds:
\begin{eqnarray}
\label{5.0}
& &\Big[\frac{N'+\tau}{p+1}-\frac{N'-4}{2} \Big] \int_{B_R} |x|^l u^{p+1} dx= \int_{|x|=R} \Big[R^{l+1} \frac{u^{p+1}}{p+1}\nonumber \\
 &&+R^{1-\alpha} \frac{v^2}{2}
+2R u'v' -R \nabla u \cdot \nabla v+\frac{N-\alpha}{2} v u'+\frac{N'-4}{2} uv'\Big]  d \sigma_R \nonumber \\
\end{eqnarray}
for all $R>0$, where $B_R=\{x \in \R^N: \; |x|<R\}$, $u'=\nabla u \cdot \frac{x}{|x|}$, $v'=\nabla v \cdot \frac{x}{|x|}$.
\end{lem}

\begin{proof}

 We know that $(u,v)$ satisfies the system:
\begin{equation}
\label{5.1}
\left \{ \begin{array}{ll} -\Delta u=|x|^{-\alpha} v \;\; &\mbox{in $\R^N \backslash \{0\}$},\\
-\Delta v=|x|^l u^p \;\; &\mbox{in $\R^N \backslash \{0\}$}.
\end{array} \right.
\end{equation}

For any $R>0$, we claim:
\begin{equation}
\label{5.2}
\int_{B_R} |\nabla u|^2 dx=\int_{B_R} |x|^{-\alpha} vu dx+\int_{|x|=R} uu' d \sigma_R<\infty,
\end{equation}
\begin{equation}
\label{5.3}
\int_{B_R} |\nabla v|^2 dx=\int_{B_R} |x|^l u^pv dx+\int_{|x|=R} vv' d \sigma_R<\infty.
\end{equation}
In particular, there exists a sequence $\epsilon_i \to 0^+$ such that
\begin{equation}
\label{5.4}
\epsilon_i \int_{|x|=\epsilon_i} |\nabla u|^2 d \sigma_{\epsilon_i} \to 0, \;\;\;
\epsilon_i \int_{|x|=\epsilon_i} |\nabla v|^2 d \sigma_{\epsilon_i} \to 0.
\end{equation}

We only show \eqref{5.2} and $\eqref{5.4}_1$, the proof of \eqref{5.3} and $\eqref{5.4}_2$ is similar. For any $0<\rho<R$, we have
\begin{eqnarray}
\label{5.5}
\int_{B_R \backslash B_\rho} |\nabla u|^2 dx &=& -\int_{B_R \backslash B_\rho} u \Delta u dx+
\int_{|x|=R} uu' d \sigma_R-\int_{|x|=\rho} u u' d \sigma_\rho \nonumber \\
&=&\int_{B_R \backslash B_\rho} |x|^{-\alpha} v u dx+
\int_{|x|=R} uu' d \sigma_R-\int_{|x|=\rho} u u' d \sigma_\rho.
\end{eqnarray}
On the other hand, we have
$$\int_{|x|=\rho} u u' d \sigma_\rho=\rho^{N-1} f'(\rho), \;\; \mbox{where
$f(\rho):=\frac{1}{2} \int_{S^{N-1}} u^2 (\rho, \theta) d \theta$}.$$
Since $f \in C^1((0,R]) \cap C^0([0,R])$ due to $u \in C^4(\R^N \backslash \{0\}) \cap C^0(\R^N)$,
we infer the existence of a sequence $\rho_i \to 0^+$ such that $\lim_{i \to \infty} \rho_i^{N-1} f'(\rho_i)=0$. Since $N \geq 5$, $\alpha<N$ and $v \in C^0(\R^N)$, passing to the limit in \eqref{5.5} with $\rho=\rho_i$, we obtain
\eqref{5.2}. Note that $\int_{B_{\rho_i}} |x|^{-\alpha} |vu| dx  \leq \frac{C}{N-\alpha} \rho_i^{N-\alpha} \to 0$ as $i \to \infty$. Hence the right-hand side of \eqref{5.2} is finite. Since
$$\int_{0}^R \int_{|x|=\epsilon} |\nabla u|^2 d \sigma_{\epsilon} d \epsilon
=\int_{B_R} |\nabla u|^2 dx,$$
our claim $\eqref{5.4}_1$ follows. Moreover,
\begin{equation}
\label{5.6}
\int_{|x|=\epsilon_i} |\nabla u| d \sigma_{\epsilon_i} \leq C\Big(\epsilon_i^{N-1} \int_{|x|=\epsilon_i} |\nabla u|^2 d \sigma_{\epsilon_i} \Big)^{\frac{1}{2}} \to 0 \;\;
\mbox{as $i \to \infty$},
\end{equation}
\begin{equation}
\label{5.7}
\int_{|x|=\epsilon_i} |\nabla v| d \sigma_{\epsilon_i} \leq C\Big(\epsilon_i^{N-1} \int_{|x|=\epsilon_i} |\nabla v|^2 d \sigma_{\epsilon_i} \Big)^{\frac{1}{2}} \to 0 \;\;
\mbox{as $i \to \infty$}.
\end{equation}

Define $w (x)=|x|^\alpha \Delta u (x)$. Then $w(x)=-v(x)$ and
$$(x \cdot \nabla u) \Delta w= (x \cdot \nabla u) |x|^l u^p=\mbox{div} \Big(x |x|^l \frac{u^{p+1}}{p+1} \Big)-\frac{N'+\tau}{p+1} |x|^l u^{p+1}.$$
Thus, for $0<\epsilon<R$, we have
$$\int_{B_R \backslash B_\epsilon}(x \cdot \nabla u) \Delta w dx
=R^{l+1} \int_{|x|=R} \frac{u^{p+1}}{p+1} d \sigma_R-\epsilon^{l+1} \int_{|x|=\epsilon} \frac{u^{p+1}}{p+1} d \sigma_\epsilon-\frac{N'+\tau}{p+1} \int_{B_R \backslash B_\epsilon} |x|^l u^{p+1} dx.$$
Letting $\epsilon \to 0$, using the continuity of $u$ and $N+l=N'+\tau>0$, we obtain
\begin{equation}
\label{5.8}
\int_{B_R} (x \cdot \nabla u) \Delta w dx=R^{l+1} \int_{|x|=R} \frac{u^{p+1}}{p+1} d \sigma_R
- \frac{N'+\tau}{p+1} \int_{B_R} |x|^l u^{p+1} dx.
\end{equation}
Next, by direct calculations, we have the following identity
\begin{equation}
\label{5.9}
\mbox{div} ((x \cdot \nabla u) \nabla w-x \nabla u \cdot \nabla w)
=(x \cdot \nabla u) \Delta w-(N-2) \nabla u \cdot \nabla w-\nabla (x \cdot \nabla w) \cdot \nabla u.
\end{equation}
It follows, for $0<\epsilon<R$,
\begin{eqnarray}
\label{5.10}
& & \int_{B_R \backslash B_\epsilon} [(x \cdot \nabla u) \Delta w-(N-2) \nabla u \cdot \nabla w-\nabla (x \cdot \nabla w) \cdot \nabla u] dx \nonumber \\
=& &\int_{|x|=R} ((x \cdot \nabla u) \nabla w-x \nabla u \cdot \nabla w) \frac{x}{|x|} d \sigma_R -\int_{|x|=\epsilon} ((x \cdot \nabla u) \nabla w-x \nabla u \cdot \nabla w) \frac{x}{|x|} d \sigma_\epsilon.\nonumber \\
\end{eqnarray}
Moreover,
\begin{eqnarray}
\label{5.11}
& & \int_{B_R \backslash B_\epsilon} \nabla u \cdot \nabla w dx \nonumber \\
=& &\int_{|x|=R} u \nabla w \cdot \frac{x}{|x|} d \sigma_R -\int_{|x|=\epsilon} u \nabla w \cdot \frac{x}{|x|} d \sigma_\epsilon-\int_{B_R \backslash B_\epsilon} u \Delta w dx \nonumber \\
=& &\int_{|x|=R} u \nabla w \cdot \frac{x}{|x|} d \sigma_R -\int_{|x|=\epsilon} u \nabla w \cdot \frac{x}{|x|} d \sigma_\epsilon-\int_{B_R \backslash B_\epsilon} |x|^l u^{p+1}dx.
\end{eqnarray}
\begin{eqnarray}
\label{5.12}
& & \int_{B_R \backslash B_\epsilon} \nabla (x \cdot \nabla w) \cdot \nabla u dx \nonumber \\
=& & \int_{|x|=R} (x \cdot \nabla w) \nabla u \cdot \frac{x}{|x|} d \sigma_R
-\int_{|x|=\epsilon} (x \cdot \nabla w) \nabla u \cdot \frac{x}{|x|} d \sigma_\epsilon
-\int_{B_R \backslash B_\epsilon} (x \cdot \nabla w) \Delta u dx  \nonumber \\
=& &\int_{|x|=R} (x \cdot \nabla w) \nabla u \cdot \frac{x}{|x|} d \sigma_R
-\int_{|x|=\epsilon} (x \cdot \nabla w) \nabla u \cdot \frac{x}{|x|} d \sigma_\epsilon
-\int_{B_R \backslash B_\epsilon} (x \cdot \nabla w) |x|^{-\alpha} w dx \nonumber \\
=& &\int_{|x|=R} (x \cdot \nabla w) \nabla u \cdot \frac{x}{|x|} d \sigma_R
-\int_{|x|=\epsilon} (x \cdot \nabla w) \nabla u \cdot \frac{x}{|x|} d \sigma_\epsilon \nonumber \\
& &-\int_{|x|=R} R^{1-\alpha} \frac{w^2}{2} d \sigma_R+\int_{|x|=\epsilon} \epsilon^{1-\alpha} \frac{w^2}{2} d \sigma_\epsilon+\frac{N-\alpha}{2}
\int_{B_R \backslash B_\epsilon} |x|^{\alpha} (\Delta u)^2 dx  \nonumber \\
=& & \int_{|x|=R} (x \cdot \nabla w) \nabla u \cdot \frac{x}{|x|} d \sigma_R
-\int_{|x|=\epsilon} (x \cdot \nabla w) \nabla u \cdot \frac{x}{|x|} d \sigma_\epsilon +\int_{|x|=\epsilon} \epsilon^{1-\alpha} \frac{w^2}{2} d \sigma_\epsilon\nonumber \\
& &  -\int_{|x|=R} R^{1-\alpha} \frac{w^2}{2} d \sigma_R  +\frac{N-\alpha}{2} \int_{|x|=R} \Big[w \nabla u \cdot \frac{x}{|x|}-u \nabla w \cdot \frac{x}{|x|} \Big] d \sigma_R \nonumber \\
& &-\frac{N-\alpha}{2} \int_{|x|=\epsilon}
\Big[w \nabla u \cdot \frac{x}{|x|}-u \nabla w \cdot \frac{x}{|x|} \Big] d \sigma_\epsilon +\frac{N-\alpha}{2} \int_{B_R \backslash B_\epsilon} |x|^l u^{p+1} dx. \nonumber\\
\end{eqnarray}

The last identity is obtained by multiplying $u$ on both the sides of (P) and integrating it on $B_R \backslash B_\epsilon$. Letting $\epsilon=\epsilon_i \to 0$, where $\epsilon_i$ is given
in \eqref{5.4}, \eqref{5.6} and \eqref{5.7}, we obtain from \eqref{5.4}, \eqref{5.6}, \eqref{5.7}, \eqref{5.10}, \eqref{5.11} and \eqref{5.12} that
\begin{eqnarray}
\label{5.13}
& &\int_{B_R} (x \cdot \nabla u) \Delta w dx \nonumber \\
=& &\int_{|x|=R} \Big[ -R^{1-\alpha} \frac{w^2}{2}+\Big((x \cdot \nabla u) \nabla w-x \nabla u \cdot \nabla w+(N-2) u \nabla w+(x \cdot \nabla w) \nabla u \nonumber \\
& &+\frac{N-\alpha}{2} (w \nabla u-u \nabla w) \Big) \cdot \frac{x}{|x|} \Big] d \sigma_R    +\Big[\frac{N-\alpha}{2}-(N-2) \Big] \int_{B_R} |x|^l u^{p+1} dx. \nonumber \\
\end{eqnarray}
Note that
\begin{eqnarray*}
\Big|\int_{|x|=\epsilon} (x \cdot \nabla w) \nabla u \cdot \frac{x}{|x|} d \sigma_\epsilon \Big|
&\leq & \epsilon \int_{|x|=\epsilon} |w'||u'| d \sigma_\epsilon \\
&\leq & \Big(\epsilon \int_{|x|=\epsilon} |\nabla w|^2 d \sigma_\epsilon \Big)^{\frac{1}{2}}
\Big(\epsilon \int_{|x|=\epsilon} |\nabla u|^2 d \sigma_\epsilon \Big)^{\frac{1}{2}},
\end{eqnarray*}
and
$$\epsilon^{1-\alpha} \int_{|x|=\epsilon} \frac{w^2}{2} d \sigma_\epsilon
=\frac{\epsilon^{N-\alpha}}{2} \int_{S^{N-1}} w^2 (\epsilon, \theta) d \theta.$$
(We know that $N-\alpha>0$ and $w$ is continuous in $B_\epsilon$.)

Combining \eqref{5.8} and \eqref{5.13}, we obtain
\begin{eqnarray}
\label{5.14}
&&\Big[\frac{N'+\tau}{p+1}-\frac{N'-4}{2} \Big] \int_{B_R} |x|^l u^{p+1} dx =\int_{|x|=R} \Big[R^{l+1} \frac{u^{p+1}}{p+1}\nonumber \\
&&+R^{1-\alpha} \frac{v^2}{2}
+2R u'v' -R \nabla u \cdot \nabla v+\frac{N-\alpha}{2} v u'+\frac{N'-4}{2} uv'\Big]  d \sigma_R. \nonumber\\
\end{eqnarray}
Note that $v=-w$.
\end{proof}

\end{document}